\documentclass[12pt]{article}
\usepackage{amsmath,amsxtra,amssymb,color,latexsym,epsfig,amscd,amsthm,subfigure,fancybox,epsfig}
\usepackage[mathscr]{eucal}
\usepackage{graphicx}
\usepackage{enumerate}
\usepackage{epsfig} 
\usepackage{epstopdf}
\usepackage{cases}
\newtheorem{thm}{Theorem}[section]
\newtheorem {asp}{Assumption}[section]
\newtheorem{lem}{Lemma}[section]

\newtheorem{prop}{Proposition}[section]
\newtheorem{deff}{Definition}[section]

\theoremstyle{definition}

\theoremstyle{remark}

\newtheorem{example}{Example}
\numberwithin{equation}{section}

\newcommand{\eps}{\varepsilon}

\newcommand{\M}{\mathcal{M}}
\newcommand{\F}{\mathcal{F}}

\newcommand{\E}{\mathbb{E}}

\newcommand{\N}{\mathbb{N}}
\newcommand{\PP}{\mathbb{P}}

\newcommand{\R}{\mathbb{R}}

\numberwithin{equation}{section}
\newcommand{\1}{\boldsymbol{1}}

\newcommand{\bed}{\begin{displaymath}}
\newcommand{\eed}{\end{displaymath}}
\newcommand{\bea}{\bed\begin{array}{rl}}
\newcommand{\eea}{\end{array}\eed}
\newcommand{\ad}{&\!\!\!\disp}
\newcommand{\aad}{&\disp}
\newcommand{\barray}{\begin{array}{ll}}
\newcommand{\earray}{\end{array}}

\def\disp{\displaystyle}

\def\bar{\overline}
\def\hat{\widehat}
\def\a.s{\text{\;a.s.\;}}

\begin{document}
\title{Coexistence and Exclusion of
Stochastic Competitive Lotka-Volterra Models\thanks{This
research was supported in part by the National Science Foundation under grant DMS-1207667.}}
\author{Dang Hai Nguyen\thanks{Department of Mathematics, Wayne State University, Detroit, MI
48202, 
dangnh.maths@gmail.com.} \and
George Yin\thanks{Department of Mathematics, Wayne State University, Detroit, MI
48202,
gyin@math.wayne.edu.}}

\maketitle

\begin{abstract} This work derives sufficient conditions
for the coexistence and exclusion of a stochastic competitive Lotka-Volterra model.
The conditions obtained are close to necessary. In addition,
convergence in distribution of positive solutions of the model is also established.
A number of numerical examples are given to illustrate our results.



\end{abstract}

\newpage




\section{Introduction}\label{sec:int}
Cooperation, predator-prey, and competition are three  main interactions
among species in eco-systems. Among them, competition is one of the
most popular interactions.
 Such interactions occur when two or more species compete for the same
 resource such as food, shelter, nesting sites, etc.
Due to competition, the growth of a species is
depressed in the presence of others.
Traditionally,
competitive interactions are modeled by systems of ordinary differential equations
known as the Lotka-Volterra models.
For instance, a competitive Lotka-Volterra model for two species takes the form
\begin{equation}\label{e1.0}
\begin{cases}
dx(t)=x(t)\big(a_1-b_1x(t)-c_1y(t)\big)dt \\
dy(t)=y(t)\big(a_2-b_2y(t)-c_2x(t)\big)dt,
\end{cases}
\end{equation}
where $x(t)$ and  $y(t)$ represent the densities of the two species at time $t$, $a_1$, and $a_2>0$ are intrinsic growth rates, and $b_1$ and $b_2>0$ are  intra-specific competition rates while $c_1$ and $c_2>0$ represent the inter-specific competition.
An important question regarding the competitive interaction is whether the species co-exist or a competitive exclusion occurs.
This question has been addressed fully for
the deterministic
model given by \eqref{e1.0}.
We state
a result whose proof can be found in \cite{HS} or \cite{Mu}.

\begin{prop}\label{thm1}
Let $\lambda_1:=a_2-c_2\dfrac{a_1}{b_1}$ and $\lambda_2:=a_1-c_1\dfrac{a_2}{b_2}$.
\begin{itemize}
\item[{\rm (i)}] If $\lambda_1>0$ and $\lambda_2>0$, all positive solutions $(x(t), y(t))$ to \eqref{e1.0} converge to the unique positive equilibrium
$\left(\dfrac{a_1c_2-a_2b_1}{c_1c_2-b_1b_2}, \dfrac{a_2c_1-a_1b_2}{c_1c_2-b_1b_2}\right)$.
\item[{\rm (ii)}] If $\lambda_1>0$ and $\lambda_2<0$, all positive solutions $(x(t), y(t))$  converge to $(0, \dfrac{a_2}{b_2})$.
\item[{\rm (iii)}] If $\lambda_1<0$ and $\lambda_2>0$, all positive solutions $(x(t), y(t))$  converge to $\left(\dfrac{a_1}{b_1}, 0\right)$.
\item[{\rm (iv)}] If $\lambda_1<0$ and $\lambda_2<0$, there is an unstable manifold $($called the separatrix$)$ splitting the interior of the positive quadrant $\R^{2,\circ}_+$ into two regions. Solutions above the separatrix  converge to $\left(0, \dfrac{a_2}{b_2}\right)$, while solutions below the separatrix tend to $\left(\dfrac{a_1}{b_1}, 0\right)$.
\end{itemize}
\end{prop}

Proposition \ref{thm1} indicates that in case (i), the interspecific competition is not too strong, so the two species coexist.
For the rest of the cases, the competitive exclusion  takes place.
In particular, in case (iv), one population with starting advantage (i.e., its initial density is sufficiently larger
than that of the other)
will eventually win, while the other will
be extinct. In addition,
In case (ii)
or (iii), one species always dominates the other.

 In the past decade, besides deterministic models, stochastic ecology models have gained increasing attention to depict more realistically
 eco-systems. The main thoughts are that such systems are often subject to environmental noise.
 Various types of
environmental noises have been considered. General Lotka-Volterra models perturbed by white noise have been studied in \cite{DS, JLO, Mx, MMR, MSR}, while the authors in \cite{LJM, LM, ZY1, ZY2} go further by considering the effect of both white and colored noises to the Lotka-Volterra models. Assuming that the population may suffer sudden environmental shocks, e.g., earthquakes, hurricanes, epidemics, etc, Bao et. al. in \cite{BMYY} consider competitive  system with jumps.
Meanwhile, Tran and Yin \cite{
TY2} use a Wonham filter to deal with a regime-switching Lotka-Volterra model in which the switching is
a hidden Markov chain.
In the aforementioned papers, some nice estimates on moment and pathwise asymptotic behaviors have been given. Some efforts have also been devoted
to providing conditions for
permanence and extinction of the species as well as the existence of stationary distribution.	
Nevertheless,
no
conditions as sharp as
their deterministic counterpart
(cf. Proposition \ref{thm1}) have been obtained.
Motivated by the needs, this paper aims to provide the classification for a stochastic competitive model that is similar to
Proposition \ref{thm1}.
Suppose that the coefficients of \eqref{e1.0} are subject to random noise that can be represented
by Brownian motions, the model becomes
\begin{equation}\label{e1.1}
\begin{cases}
dX(t)=X(t)\big(a_1-b_1X(t)-c_1Y(t)\big)dt+(\alpha_1X^2(t)+\gamma_1X(t))dB_1(t)+\beta_1 X(t)Y(t)dB_2(t), \\
dY(t)=Y(t)\big(a_2-b_2Y(t)-c_2X(t)\big)dt+(\alpha_2Y^2(t)+\gamma_2Y(t))dB_3(t)+\beta_2 X(t)Y(t)dB_2(t),
\end{cases}
\end{equation}
where $B_1(\cdot)$, $B_2(\cdot)$, and $B_3(\cdot)$ are independent Brownian motions.
To reduce unnecessary computations due to notational complexity and to make our ideas more understandable but still
 preserve important properties, we assume that the lowest-power terms are not affected by environment noise
 for simplicity, that is, $\gamma_1=\gamma_2=0$. Thus, the following model will be considered throughout the rest of the paper:
\begin{equation}\label{e1.2}
\begin{cases}
dX(t)=X(t)\big(a_1-b_1X(t)-c_1Y(t)\big)dt+\alpha_1X^2(t)dB_1(t)+\beta_1 X(t)Y(t)dB_2(t), \\
dY(t)=Y(t)\big(a_2-b_2Y(t)-c_2X(t)\big)dt+\alpha_2Y^2(t)dB_3(t)+\beta_2 X(t)Y(t)dB_2(t).
\end{cases}
\end{equation}
Similar to the deterministic case, we introduce two values $\lambda_1,\lambda_2$ that can be
 considered as threshold values and that can be
 calculated from the coefficients.
We show that if both of them are positive, the
coexistence takes place and all positive solutions to \eqref{e1.2} converge to a unique invariant probability measure in total variation norm.
If one of the quantities is positive and the other is negative, then one species will
dominate, the other will die out.
In
case both values are negative, each species will die out with a positive probability.
Another distinctive contribution of this paper is the demonstration of  link of the threshold values and the Lyapunov exponents.
 We
 demonstrate that when $Y(t)$ or $X(t)$ converge to $0$, their Lyapunov exponents are
 precisely $\lambda_1$ and $\lambda_2$, respectively.
It should be mentioned that some
 related results have been obtained for stochastic Lotka-Volterra models of predator-prey type; see \cite{DDT, RR}.
However, the methods used in \cite{DDT, RR} are not applicable to our
model
 for two reasons.
First, relying on the basic principle that the predator will die out without prey, there is only one threshold value determining whether the predator
 will be extinct or permanent. In contrast, our model requires to examine two values arising from the behavior of solutions
  leading to much more difficulty. Second, in \cite{DDT, RR}, the inter-specific terms were assumed not
 to
 subject to random noise so that the solutions in $\R^{2,\circ}_+$, the interior of $\R^2_+$, can be compared easily to the solutions on the boundary.
It is not the case for our model. Some new techniques will therefore be introduced to overcome the difficulty.
Moreover, it can be seen in our proofs  that similar results can be obtained for the general model \eqref{e1.1} using our new method.

To proceed,
the rest of the paper is arranged as follows.
We present our main results and provide some numerical examples demonstrating our findings in Section \ref{sec:mai}.
Because the proofs are quite technical,
Sections \ref{sec:coe} and  \ref{sec:com} are devoted to the proofs for the coexistence and the exclusion cases, respectively.
In Section \ref{sec:pie}, we treat a Kolmogorov system of competitive type under telegraph noise.
That section complements our earlier results in \cite{DDY1}.
We conclude with discussion on model \eqref{e1.1} and its variants.

\section{Main Results}\label{sec:mai}
Let $(\Omega,\F,\{\F_t\}_{t\geq0},\PP)$ be a complete filtered probability space with the filtration $\{\F_t\}_{t\geq 0}$ satisfying the usual condition,  i.e., it is increasing and right continuous while $\F_0$ contains all $\PP$-null sets.
We consider
model \eqref{e1.2},
where
 $B_1(t)$, $B_2(t)$, and $B_3(t)$ are three $\F_t$-adapted, mutually independent Brownian motions.
We suppose that $a_i, b_i, c_i$ are positive constants for $i=1,2$.
We also suppose that $\alpha_i\ne 0$, $i=1,2$ so that the diffusion is non-degenerate.
The degenerate case will be discussed later.
Throughout this paper, to simplify the notation, we denote $z=(x,y), z_0=(x_0, y_0)$, and $Z(t)=(X(t), Y(t)).$
We also denote $a\wedge b=\min\{a,b\}$, $a\vee b=\max\{a,b\}$, and $\R^{2,\circ}_+=\{(x,y): x>0,y>0\}$.
Let $Z_z(t)=(X_z(t), Y_z(t))$ be the solution to \eqref{e1.2} with initial value $z$.
It is proved in \cite{MMR} that if $z\in\R^{2,\circ}_+$, $Z_z(t)$ remains in $\R^{2,\circ}_+$ with probability 1.
Moreover, the solution $Z(t)$ is a strong homogeneous Markov process. We state some important properties of the solution whose proof can be found in \cite{MMR, MSR, LJM}.

\begin{prop}\label{prop2.1}
The following assertions hold:
\begin{itemize}
\item[{\rm(i)}]
There is an $M_0>0$ such that $$\limsup\limits_{t\to\infty} \E V
(X_{z }(t),Y_{z }(t))\leq M_0\,\forall z\in\R^2_+\setminus\{(0,0)\}$$ where $V(x,y)=(x+y)^{-1}+(x+y).$
  \item[{\rm (ii)}] For any  $\eps>0$, $H>1$, $T>0$, there is an $\bar H=\bar H(\eps, H, T)>1$ such that
  $$\PP\left\{\bar H^{-1}\leq X_z(t)\leq \bar H\,\forall t\in[0,T]\right\}\geq1-\eps \ \hbox{ if } \ z\in[H^{-1},H]\times[0,H]$$ and that
$$\PP\left\{\bar H^{-1}\leq Y_z(t)\leq \bar H\,\forall t\in[0,T]\right\}\geq1-\eps \ \hbox{ if } \ z\in[0,H]\times[H^{-1},H].$$
\item[{\rm (iii)}] For any $p\in(0,3)$, there is an $M_{p}>0$ such that $$\E\int_0^t\|Z_{z }(s)\|^{p}\leq  M_{p}(t+\|z \|)\,\forall z \in\R^2_+, t\geq 0.$$
\end{itemize}
\end{prop}

To take an
in-depth
study,  we first consider the equation on the boundary.
On the $x$-axis, we have
\begin{equation}\label{e2.1}
d\varphi (t)=\varphi (t)\big(a_1-b_1\varphi (t)\big)dt+\alpha_1\varphi^2 (t)dB_1(t).
\end{equation}
This
diffusion has a unique invariant probability measure $\pi^*_1$ in $(0,\infty)$ with density
$$f^*_1(\phi)=\dfrac{c^*_1}{\phi^4}\exp\left(\dfrac{2b_1}{\alpha_1^2}\dfrac1\phi-\dfrac{a_1}{\alpha_1^2}\dfrac1{\phi^2}\right), \phi>0$$
where $c^*_1$ is the normalizing constant.
We refer to \cite{DDT} for the proof and the expression of $c^*_1$.
By the ergodicity (see \cite[Theorem 3.16]{AS}),
for any measurable function $h(\cdot):\R_+\to\R$ satisfying that $\int_0^\infty |h(\phi)|f^*_1(\phi)d\phi<\infty$, we have
\begin{equation}\label{erg}
\PP\left\{\lim\limits_{T\to\infty}\dfrac1T\int_0^Th(\varphi_{x}(t))dt=\int_0^\infty h(\phi)f^*_1(\phi)d\phi\right\}=1\,\forall x>0,
\end{equation}
where $\varphi_{x}$ is the solution to \eqref{e2.1} starting at $x$.
In particular, for any $p\in(-\infty,3)$,
\begin{equation}\label{e2.2}
\PP\left\{\lim\limits_{T\to\infty}\dfrac1T\int_0^T\varphi_{x}^p(t)dt=Q_p:=\int_0^\infty \phi^pf^*_1(\phi)d\phi<\infty\right\}=1\,\forall x>0.
\end{equation}
We define
\begin{equation}\label{e2.3}
\lambda_1=\int_0^\infty\left(a_2-c_2\phi-\dfrac{\beta_2^2}2\phi^2\right)f^*_1(\phi)d\phi=a_2-c_2Q_1-\dfrac{\beta_2^2}2Q_2.
\end{equation}
Similarly, considering
diffusion whose  equation on the $y$-axis is
$$d\psi (t)=\psi (t)\big(a_2-b_2\psi (t)\big)dt+\alpha_2\psi^2 (t)dB_3(t),$$
which has a unique invariant probability measure $\pi^*_2$. We can define
\begin{equation}\label{e2.4}
\lambda_2=\int_0^\infty\left(a_1-c_1\phi-\dfrac{\beta_1^2}2\phi^2\right)f^*_2(\phi)d\phi,
\end{equation}
where $f^*_2(\cdot)$ is the density function of $\pi^*_2$ given by
$$f^*_2(\phi)=\dfrac{c^*_2}{\phi^4}\exp\left(\dfrac{2b_2}{\alpha_2^2}\dfrac1\phi-\dfrac{a_2}{\alpha_2^2}\dfrac1{\phi^2}\right), \phi>0.$$
Let us
elaborate on the definition and use of  $\lambda_1$ and $\lambda_2$.
To determine whether  $Y_z(t)$ converges to 0 or not, we consider the Lyapunov exponent of $Y_z(t)$ when $Y_z(t)$ is small for a sufficiently long time. Hence, we look at the following equation which is derived from It\^o's formula.
\begin{equation}\label{e2.5}
\begin{aligned}
\dfrac{\ln Y_z(T)}T=\dfrac{\ln y}T&+\dfrac1T\int_0^T\left(a_2-b_2Y_z(t)-\dfrac{\alpha_2^2}2Y_z^2(t)-c_2X_z(t)-\dfrac{\beta_2^2}2X_z^2(t)\right)dt\\
&+\dfrac1T\int_0^T\Big(\alpha_2 Y_z(t)dB_3(t)+\beta_2X_z(t)dB_2(t)\Big).
\end{aligned}
\end{equation}
When $T$ is large, the first and third terms on the right-hand side of \eqref{e2.5} are small.
Intuitively, if $Y_z(t)$ is small in $[0,T]$, $X_z(t)$ is close to $\varphi_x(t)$.
Using the ergodicity, we see that $\dfrac{\ln y(T)}T$ is close to $\lambda_1$.
We here give the definitions of stochastic coexistence and competitive exclusion and then states our main results whose proofs are left to Sections \ref{sec:coe} and \ref{sec:com}.

\begin{deff} {\rm
The populations of two species modeled by \eqref{e1.2} are said to stochastically coexist if for any $\eps>0$, there is an $M=M(\eps)>1$ such that
$$\liminf\limits_{t\to\infty}\PP\left\{M^{-1}\leq X(t), Y(t)\leq M\right\}\geq 1-\eps.$$
The competitive exclusion is said to
take place
almost surely if $$\PP\left\{\lim\limits_{t\to\infty}X(t)=0 \text{ or }\lim\limits_{t\to\infty}Y(t)=0\right\}=1.$$
}\end{deff}

\begin{thm}\label{thm2.1}
If $\lambda_1$ and $\lambda_2$ are both positive, the two species coexist. Moreover, there is a unique invariant measure $\mu^*$ with support $\R^{2,\circ}_+$
of the solution process $Z(t)$ such that
\begin{itemize}
  \item[{\rm (i)}] the transition probability $P(t, z,\cdot)$ of $Z(t)$ converges in total variation to $\mu^*$  $\,\forall\,z\in \R^{2\circ}_+$;
  \item[{\rm (ii)}] for any $\mu^*$-integrable function $F(z):\R^{2,\circ}_+\to\R$, we have
$$\lim\limits_{t\to\infty}\dfrac1t\int_0^tF(Z_{z_0}(s))ds=\int_{\R^{2\circ}_+}F(z)\mu^*(dz)\text{ a.s. }\,\forall z_0\in \R^{2\circ}_+.$$
\end{itemize}
\end{thm}

The following two theorems give
criteria under which
the competitive exclusion takes place almost surely.

\begin{thm}\label{thm2.2}
If $\lambda_1<0$ and $\lambda_2>0$ then the distribution of $X_{z_0}(t)$ converges weakly to $\pi^*_1$ and
$$\PP\left\{\lim\limits_{t\to\infty} \dfrac{\ln Y_{z_0}(t)}t=\lambda_1<0\right\}=1\,\forall z_0\in\R^{2\circ}_+.$$
If $\lambda_1>0$ and $\lambda_2<0$ then the distribution of $Y_{z_0}(t)$ converges weakly to $\pi^*_2$ while
$$\PP\left\{\lim\limits_{t\to\infty} \dfrac{\ln X_{z_0}(t)}t=\lambda_2<0\right\}=1\,\forall z_0\in\R^{2\circ}_+.$$
\end{thm}

\begin{thm}\label{thm2.3}
Suppose that $\lambda_1$ and $\lambda_2$ are both negative. For any $z_0\in\R^{2\circ}_+$, we have
$p_{z_0}>0, q_{z_0}>0$ and $p_{z_0}+q_{z_0}=1$ where
$$p_{z_0}=\PP\left\{\lim\limits_{t\to\infty} \dfrac{\ln X_{z_0}(t)}t=\lambda_2\right\}\text{ and } q_{z_0}=\PP\left\{\lim\limits_{t\to\infty} \dfrac{\ln Y_{z_0}(t)}t=\lambda_1\right\}.$$
Moreover the distribution of $Z_{z_0}(t)$ converges weakly to $\mu_{z_0}:=p_{z_0}(\delta^*\times\pi^*_2)+q_{z_0}(\pi^*_1\times\delta^*)$
where
$\delta^*$ is the Dirac measure concentrated at $0$.
To be more precise, for any measurable sets $A, B\subset\R$, $\mu_{z_0}(A\times B)=p_{z_0}\delta^*(A)\pi^*_2(B)+q_{z_0}\pi^*_1(A)\delta^*(B)$.		
\end{thm}

\begin{example}\label{ex1}
Consider \eqref{e1.2} with parameters $a_1=4, a_2=3$, $b_1=1.5, b_2=1$, $c_1=1, c_2=0.5$, $\alpha_1=0.25, \alpha_2=0.5$, $\beta_1=0.5, \beta_2=0.25$.
Direct calculation shows that $\lambda_1=1.08, \lambda_2=1.53$.  In view of Theorem \ref{thm2.1}, \eqref{e1.2} has a unique invariant probability measure $\mu^*$ with support $\R^{2,\circ}_+$.
Moreover, the strong law of large numbers and the convergence in total variation of the transition probability hold.
We provide Figure \ref{f1.1} for illustration.
\begin{figure}
\centering
\includegraphics[totalheight=2in]{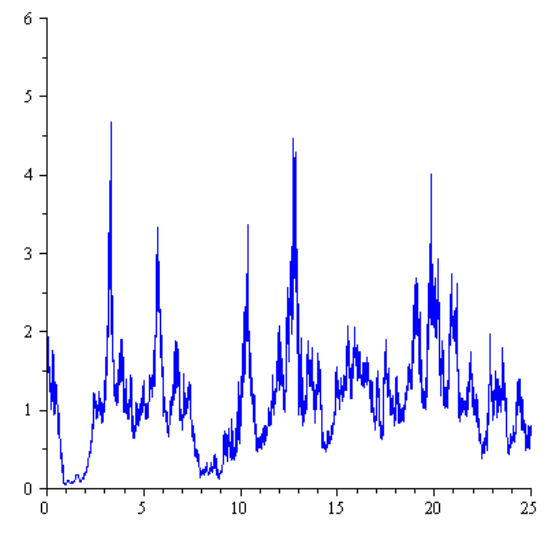}
\includegraphics[totalheight=2in]{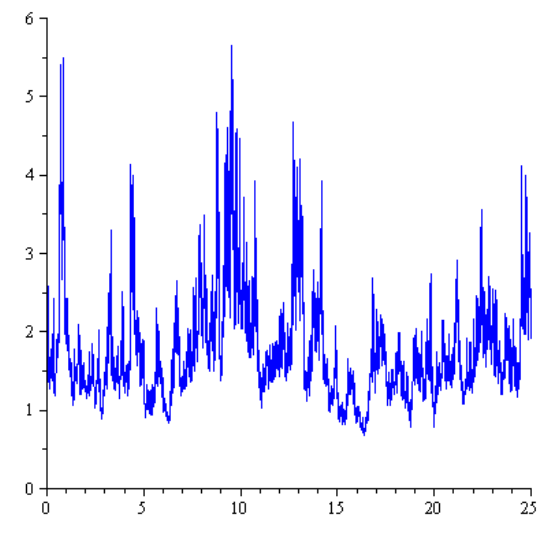}
\includegraphics[totalheight=2in]{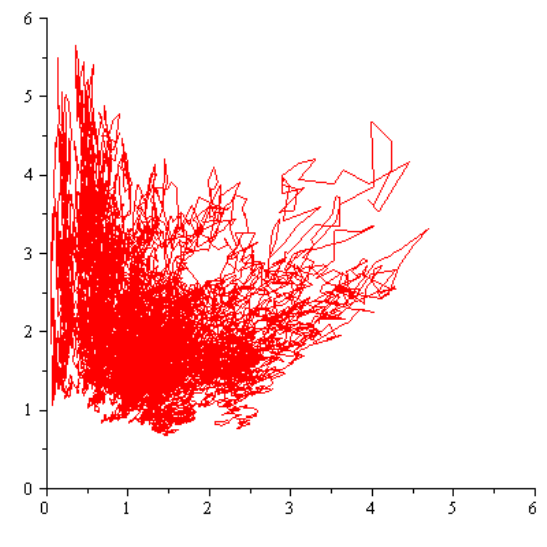}
\caption{\label{f1.1}
Trajectories of $X_z(t), Y_z(t)$ and phase portrait $(X_z(t), Y_z(t))$ in Example \ref{ex1} with $z=(2,2)$.}
\end{figure}
\end{example}
\begin{example}\label{ex2}
Consider \eqref{e1.2} with parameters $a_1=4, a_2=2$, $b_1=1.5, b_2=1$, $c_1=2, c_2=1$, $\alpha_1=1, \alpha_2=0.5$, $\beta_1=0.5, \beta_2=1$.
In this example, $\lambda_1=-1.07, \lambda_2=0.41$.  In view of Theorem \ref{thm2.2}, $\lim\limits_{t\to\infty}\dfrac{\ln Y_z(t)}t= -1.07$ while $X_z(t)$ converges in distribution to $\pi^*_1$.
Sample paths of $\ln X_z(t), \ln Y_z(t)$ and phase portrait of $(X_z(t), Y_z(t))$ with $z=(2,2)$ are
plotted in Figure \ref{f1.2}.
\begin{figure}
\centering
\includegraphics[totalheight=2in]{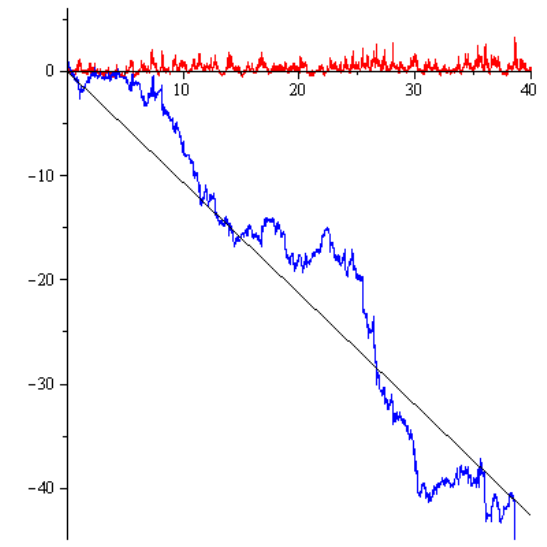}
\includegraphics[totalheight=2in]{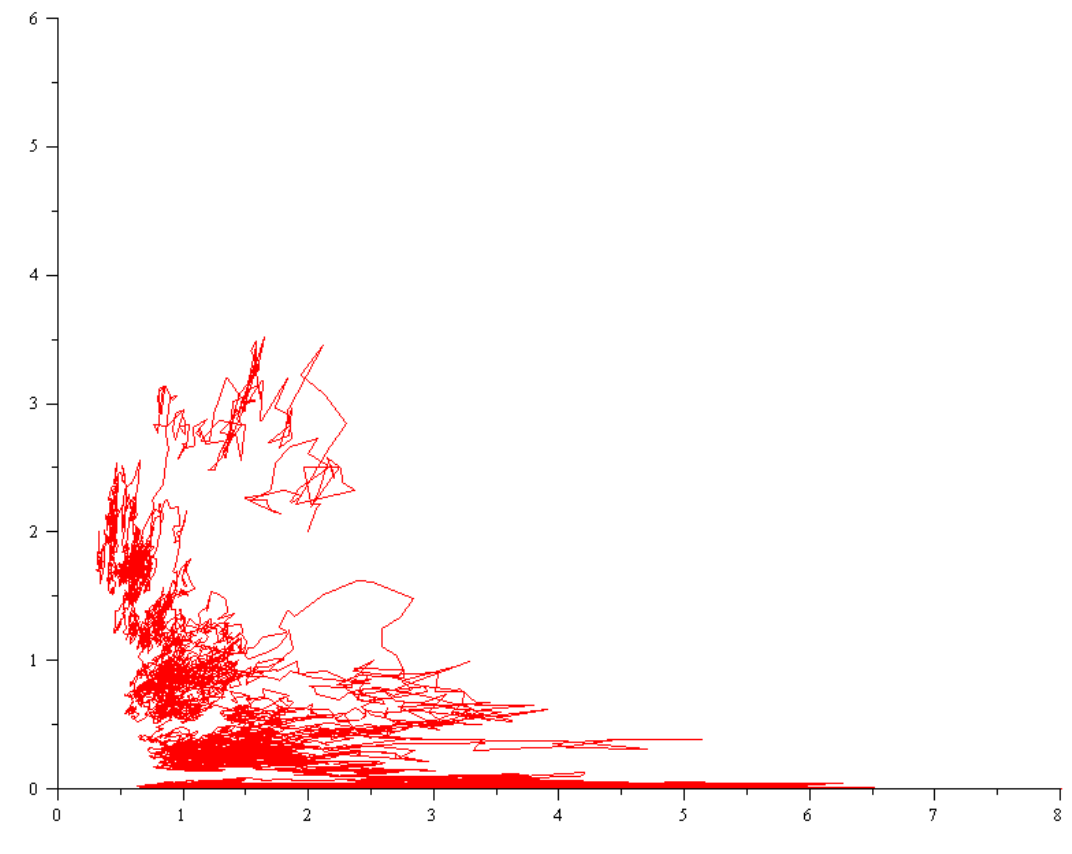}
\caption{\label{f1.2}
Sample paths of $\ln (X_z(t))$ (in red) and $\ln (Y_z(t))$ (in blue) and phase portrait of Example \ref{ex2}.}
\end{figure}
\end{example}

\begin{example}\label{ex3}
Consider \eqref{e1.2} with parameters $a_1=a_2=2$, $b_1=b_2=1$, $c_1=c_2=2$, $\alpha_1=\alpha_2=1$, $\beta_1=\beta_2=1$.
We have $\lambda_1=\lambda_2=-1.06$.  This system is symmetric. The initial value has the same coordinates: $z=(2,2)$.
Hence, the probabilities that the solution converges to the two axes are the same.
We provided two trials. One of them results in the convergence to the $y$-axis. The other shows the convergence to the $x$-axis. Figures \ref{f1.3}, \ref{f1.4} validate our claim.
\begin{figure}
\centering
\includegraphics[totalheight=1.9in]{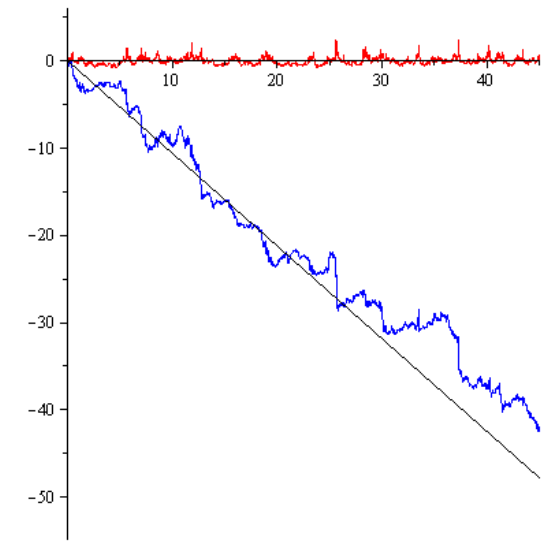}
\includegraphics[totalheight=1.9in]{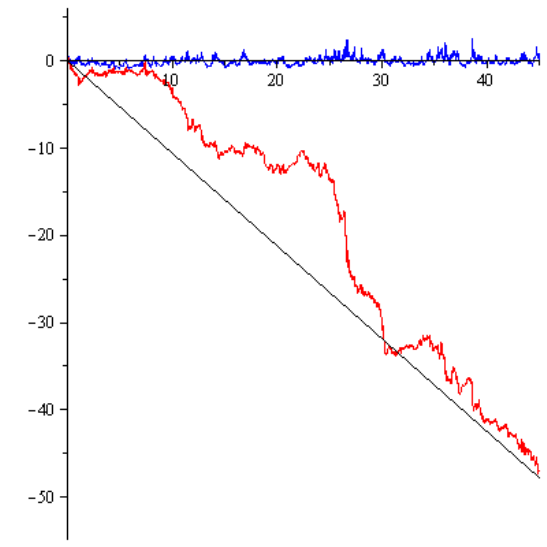}
\caption{\label{f1.3}
Sample paths of $\ln (X_z(t))$ (in red) and $\ln (Y_z(t))$ (in blue) of Ex. \ref{ex3} in two trials. The black line has slope $\lambda_1=\lambda_2<0$.}
\end{figure}
\begin{figure}
\centering
\includegraphics[totalheight=2in]{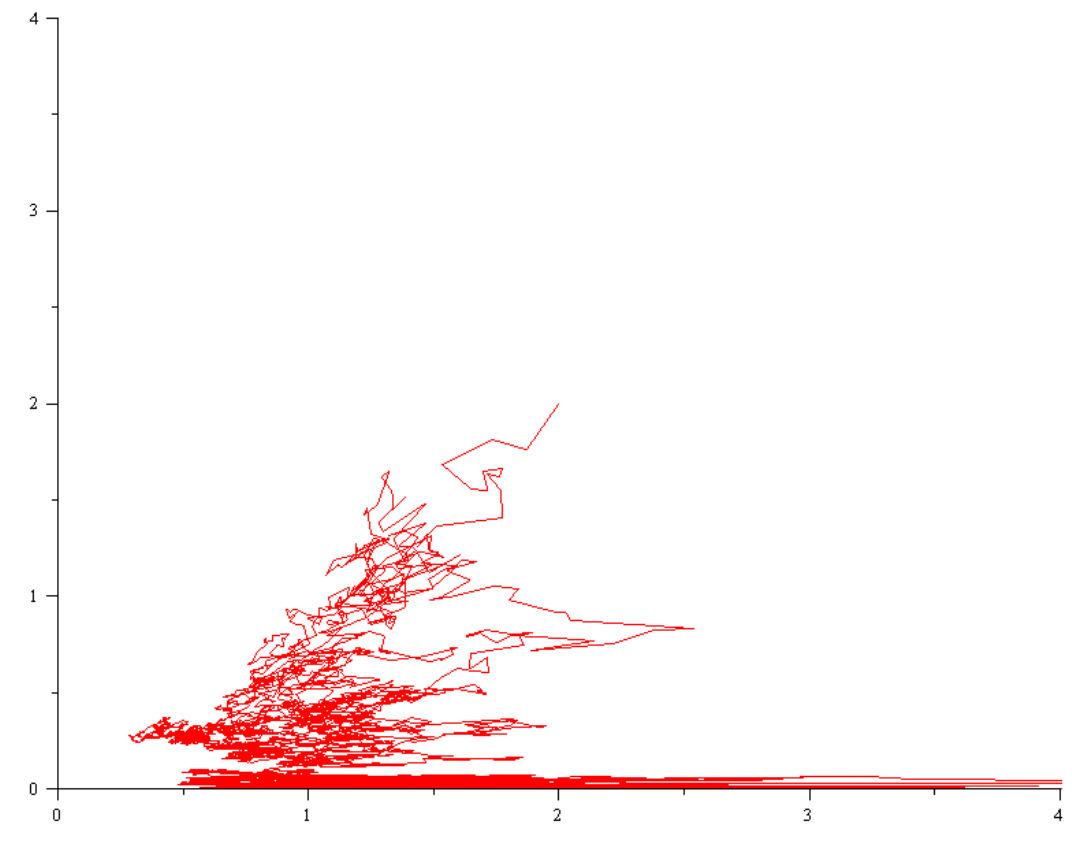}
\includegraphics[totalheight=2in]{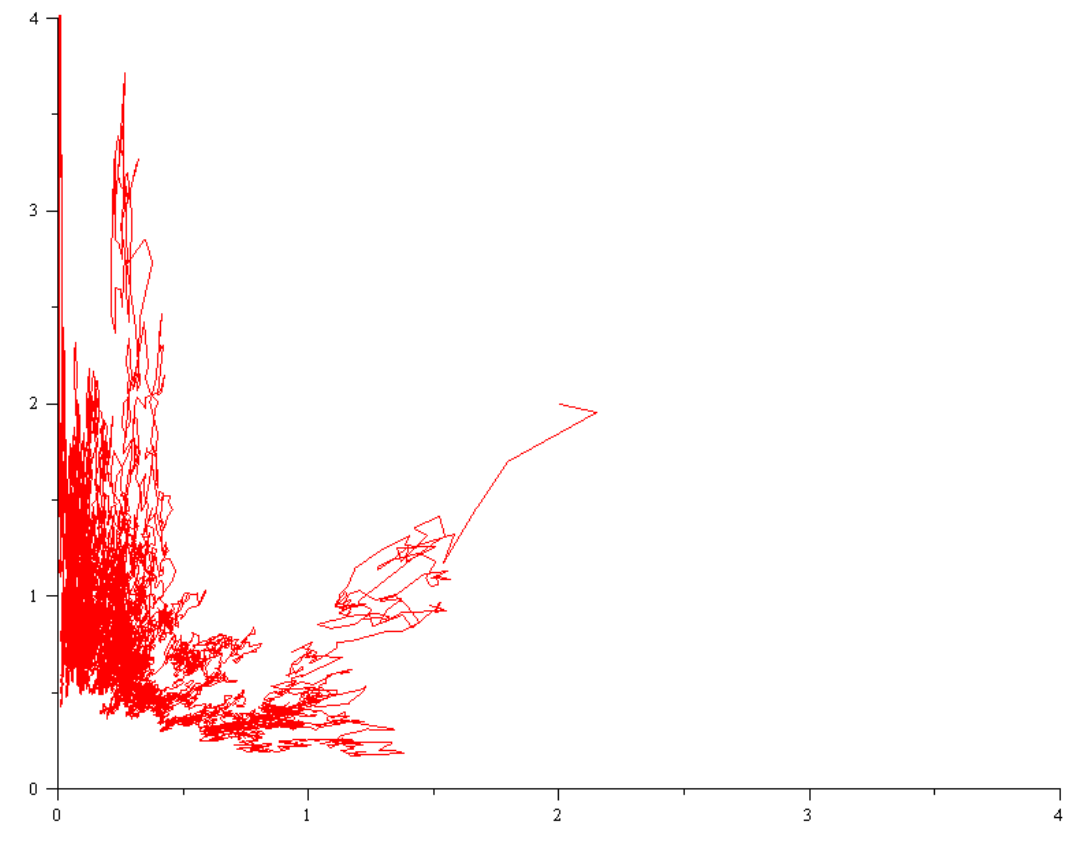}
\caption{\label{f1.4}
Phase portraits of the solution of Ex. \ref{ex3} in two trials.}
\end{figure}
\end{example}

\section{Coexistence}\label{sec:coe}
This section is devoted to proving Theorem \ref{thm2.1}.
The following formula is the well-known  exponential martingale inequality, which will be used several times in our proofs.
It asserts that for any $a, b>0$,
\begin{equation}\label{emi}
\PP\left\{\int_0^tg(s)dW(s)-\dfrac{a}{2}\int_0^tg^2(s)ds>b\,\forall t\geq0\right\}\leq e^{-ab},
\end{equation}
if $W(t)$ is a $\F_t$-adapted Brownian motion while $g(t)$ is a real-valued $\F_t$-adapted process and $\int_0^tg^2(s)ds<\infty\,\forall t\geq0$ almost surely (see \cite[Theorem 1.7.4]{XM}).
It should be noted that  in  \cite[Theorem 1.7.4]{XM}, the inequality is stated for a finite interval.
However, \eqref{emi} holds since
$$
\begin{aligned}
&\PP\left\{\int_0^tg(s)dW(s)-\dfrac{a}{2}\int_0^tg^2(s)ds>b\,\forall t\geq0\right\}\\
&\qquad\qquad\qquad=\lim\limits_{T\to\infty}\PP\left\{\int_0^tg(s)dW(s)-\dfrac{a}{2}\int_0^tg^2(s)ds>b\,\forall \,t\in[0,T]\right\}.
\end{aligned}
$$
Let any $T>1$,
 $p^*\in (1,1.5)$, and $\dfrac1{p^*}+\dfrac1{q^*}=1$. For $A\in\F$, denote by $\1_{A}$ the indicator function of $A$.
Using part (iii) of Proposition \ref{prop2.1} and Holder's inequality, we can estimate
\begin{equation}\label{e3.1}
\begin{aligned}
\E\1_A\big|\ln Y_{z }(T)-\ln y \big|\leq& \E\int_0^T\1_A\left|a_2+b_2Y_{z }(s)+c_2X_{z }(s)+\dfrac{\alpha_2^2Y_{z }^2(s)+\beta_2^2X_{z }^2(s)}{2}\right|ds\\
&+\E\1_A\left|\int_0^T(\alpha_2Y_{z }dB_3(s)+\beta_2X_{z }(s))dB_2(s)\right|\\
\leq&\theta_1\left(\PP(A)T\big)^{1/q^*}\Big(\E\int_0^T\big(1+X_{z }^{2p^*}(s)+Y_{z }^{2p^*}(s)\big)ds\right)^{1/p^*}\\
&+\sqrt{\PP(A)}\left(\E\int_0^T(\alpha_2^2Y_z^2+\beta^2_2X^2_{z }(s))ds\right)^{\frac12}\\
\leq&\theta_2\big(\PP(A)\big)^{1/q^*}(1+|z |)^{1/p^*}T\qquad (\text{ since } 1/q^*<1/2).
\end{aligned}
\end{equation}
for some constants $\theta_1,\theta_2$ independent of $z$, $T$ and $A$.
In particular, when $A=\Omega$,
\begin{equation}\label{e3.2}
\E\Big|\dfrac{\ln Y_{z }(T)-\ln y}T \Big|\leq \theta_2(1+|z |)^{1/p^*}.
\end{equation}
and consequently,
\begin{equation}\label{e3.3}
\PP\left\{\Big|\dfrac{\ln Y_{z }(T)-\ln y}T \Big|\geq \dfrac{\theta_2(1+|z |)^{1/p^*}}\eps\right\}\leq\eps.
\end{equation}
In what follows, we define the stopping time
$$\tau_{z }^\sigma =\inf\{t\geq0: Y_z(t)\geq\sigma\}.$$

\begin{lem}\label{lm3.2}
For any $T>1,\eps>0, \sigma>0$, there is a $\delta=\delta(T, \eps, \sigma)>0$
such that $$\PP\{\tau_{z }^\sigma \geq T\}\geq 1-\eps\, \ \forall z\in(0,\infty)\times(0,\delta].$$
\end{lem}

%
%

\begin{proof}
By the exponential martingale inequality,
$\PP(\Omega_1^z)\geq1-\eps	$,
where
$$\Omega_1^z=\left\{\int_0^{t}(\alpha_2Y_zdB_3(s)+\beta_2X_{z }(s)dB_2(s))<\dfrac12\int_0^{t}\big(\alpha_2^2Y^2_z(s)+\beta^2_2X^2_{z}(s)\big)ds+\ln\dfrac1\eps\,\forall t\right\}.$$
In view of \eqref{e2.5},
when $\omega\in \Omega_1^z$ we have
$$\ln Y_{z } (t)<\ln y +\ln\dfrac1\eps+\int_0^{t}a_2dt=\ln y+\ln\dfrac1\eps+a_2t\, \ \forall t\geq0.$$
Letting $\delta=\sigma\eps e^{-a_2T}$, we can see that if $y\leq\delta$, then
$Y_{z } (t)<\sigma\,\forall t<T, \omega\in\Omega_1^z$.
\end{proof}

\begin{lem}\label{lm3.3}
For any $H, T>1,\eps, \nu >0$, there is a $\sigma>0$ such that for all $z \in [H^{-1},H]\times(0,\sigma]$,
$$\PP\{|\varphi_{x }(t)-X_{z }(t)|<\nu\,\forall 0\leq t\leq T\wedge\tau_{z }^\sigma \}\geq 1-\eps.$$
\end{lem}

\begin{proof}
By part (ii) of Proposition \ref{prop2.1}, we can find $\bar H$ sufficiently large such that
$$\PP\{(\varphi_{x }(t)\big)\vee \big(X_{z }(t)\big)\leq \bar H\,\forall t\leq T\}\geq 1-\dfrac\eps2\,\forall z\in [H^{-1},H]\times(0,1].$$
Let $\xi_{z }:=\tau_{z }^\sigma \wedge\inf\big\{u: \big(\varphi_{x }(u)\big)\vee \big(X_{z}(u)\big)\geq \bar H\big\}.$
It follows from the It\^o formula that
$$
\begin{aligned}
|\varphi_{x }(s)-X_{z }(s)|\leq&\int_0^s|\varphi_{x }(u)-X_{z }(u)|\big(a_1+b_1(\varphi_{x }(u)+X_{z }(u))\big)du\\
&+c_1\int_0^s X_{z }(u)Y_{z }(u)du+|\beta_1|\left|\int_0^sX_{z }(u)Y_{z }(u)dB_2(u)\right|\\
&+|\alpha_1|\left|\int_0^s\big(\varphi_{x }(u)-X_{z }(u)\big)(\varphi_{x }(u)+X_{z }(u)\big)dB_1(u)\right|.
\end{aligned}
$$
The elementary inequality $\big(\sum_{i=1}^n a_i\big)^2\leq 2^n\sum_{i=1}^n a^2_i$ leads to
\begin{equation}\label{e3.6}
\begin{aligned}
\E\sup\limits_{s\leq t}\big(&\varphi_{x }(t\wedge\xi_{z })-X_{z }(t\wedge\xi_{z })\big)^2\\
\leq&16\E \left(\int_0^{t\wedge\xi_{z }}|\varphi_{x }(u)-X_{z }(u)|\big(a_1+b_1(\varphi_{x }(u)+X_{z }(u)\big)du\right)^2\\
&+16c^2_1\E\int_0^{t\wedge\xi_{z }} X^2_{z }(u)Y^2_{z }(u)du+16\beta_1^2\E\sup\limits_{s\leq t}\left|\int_0^{s\wedge\xi_{z }}X_{z }(u)Y_{z }(u)dB_2(u)\right|^{2}\\
&+16\alpha_1^2\E\sup\limits_{s\leq t}\left|\int_0^{s\wedge\xi_{z }}\big(\varphi_{x }(u)-X_{z }(u)\big)(\varphi_{x }(u)+X_{z }(u)\big)dB_1(u)\right|^2.
\end{aligned}
\end{equation}
We have the following estimates for $t\in[0,T]$
\begin{equation}\label{e3.7}
\E\sup\limits_{s\leq t}\left|\int_0^{s\wedge\xi_{z }}X_{z }(u)Y_{z }(u)dB_2(u)\right|^{2}\leq 4\sigma^2\E\int_0^{t\wedge\xi_{z }}X^2_{z }(u)du\leq 4\bar H^2T\sigma^2,
\end{equation}
\begin{equation}\label{e3.8}
\E\int_0^{t\wedge\xi_{z }} X^2_{z }(u)Y^2_{z }(u)du\leq \bar H^2\sigma^2T,
\end{equation}
\begin{equation}\label{e3.9}
\begin{aligned}
\E\sup\limits_{s\leq t}\bigg|&\int_0^{s\wedge\xi_{z }}\big(\varphi_{x }(u)-X_{z }(u)\big)(\varphi_{x }(u)+X_{z }(u)\big)dB_1(u)\bigg|^2\\
\leq& 4\E\int_0^{t\wedge\xi_{z }}\big(\varphi_{x }(u)-X_{z }(u)\big)^2(\varphi_{x }(u)+X_{z }(u)\big)^2d(u)\\
\leq& 16\bar H^2\E\int_0^{t\wedge\xi_{z }}\big(\varphi_{x }(u)-X_{z }(u)\big)^2du,
\end{aligned}
\end{equation}
where \eqref{e3.7} and \eqref{e3.9} follow from the Burkholder-Davis-Gundy inequality. 	
By Holder's inequality,
\begin{equation}\label{e3.10}
\begin{aligned}
\E \bigg(\int_0^{t\wedge\xi_{z }}|&\varphi_{x }(u)-X_{z }(u)|\big(a_1+b_1(\varphi_{x }(u)+X_{z }(u))\big)du\bigg)^2\\
&\leq (a_1+2b_1\bar H)^2T \E\int_0^{t\wedge\xi_{z }}\big(\varphi_{x }(u)-X_{z }(u)\big)^2du\,\forall\,t\in[0,T].
\end{aligned}
\end{equation}
Applying \eqref{e3.7}, \eqref{e3.8}, \eqref{e3.9}, and \eqref{e3.10} to \eqref{e3.6} we have
$$
\begin{aligned}
\E\sup\limits_{s\leq t}&\big(\varphi_{x }(s\wedge\xi_{z })-X_{z }(s\wedge\xi_{z })\big)^2\\
&\leq \bar m\left(\sigma^2+\E\int_0^{t\wedge\xi_{z }}\big(\varphi_{x }(u)-X_{z }(u)\big)^2du\right)\\
&\leq \bar m\left(\sigma^2+\int_0^{t}\Big(\E\sup\limits_{s\leq u}\big(\varphi_{x }(s\wedge\xi_{z })-X_{z }(s\wedge\xi_{z })\big)^2\right)du\,\forall\,t\in[0,T]
\end{aligned}
$$ for some $\bar m=\bar m(\bar H, T)>0.$
Applying Gronwall's inequality,
$$\E\sup\limits_{s\leq T}\big(\varphi_{x }(s\wedge\xi_{z })-X_{z }(s\wedge\xi_{z })\big)^2\leq \bar m\sigma^2\exp(\bar m T).$$
As a result,
$$\PP\left\{\sup\limits_{s\leq T}\big(\varphi_{x }(s\wedge\xi_{z })-X_{z }(s\wedge\xi_{z })\big)^2\geq\nu^2\right\}\leq  \dfrac{\bar m\sigma^2e^{\bar m T}}{\nu^{2}}<\dfrac\eps2\text{ when } \sigma \text{ is sufficiently small}.$$ Then
$$\PP\left\{s\wedge\xi_{z }=s\wedge\tau_z^\sigma\,\forall s\in[0,T]\right\}\geq\PP\left\{ \sup\limits_{s\leq T}\big\{\big(\varphi_{x }(s)\big)\vee \big(X_{z }(s)\big)\big\}
\leq \bar H\right\}\geq 1-\dfrac{\eps}2,$$
yielding the desired result.
\end{proof}

\begin{lem}\label{lm3.4} For any $\eps>0$,
there is an $\hat M>0$ such that
$$\PP\left\{\left|\int_0^T\alpha_2 Y_{z }(t)dB_3(t)+\beta_2 X_{z }(t)dB_2(t)\right|\leq \dfrac{\hat M}{\eps}\sqrt{T\|z \|}\right\}\geq 1-\eps.$$
\end{lem}

\begin{proof}
Since $$\E \left|\int_0^T\alpha_2 Y_{z }(t)dB_3(t)+\beta_2 X_{z }(t)dB_2(t)\right|^2=\E\int_0^T\big(\alpha_2^2Y^2_{z }(t)+\beta_2^2 X^2_{z }(t)\big)dt,$$ using (iii) of Proposition \ref{prop2.1} and Chebyshev's inequality we obtain the result.
\end{proof}

\begin{prop}\label{prop3.1}
Assume that $\lambda_1>0$.
For any $\eps>0,H>1$, there are $T=T(\eps, H)>0$ and $\delta_0=\delta_0(\eps, H)$ satisfying that for any $z \in [H^{-1},H]\times(0,\delta_0]$,
$\PP(\hat\Omega^{z })>1-4\eps$, where
$$\hat\Omega^{z }=\left\{\dfrac{\lambda_1}5T\leq\ln Y_{z }(T)-\ln y \right\}.$$
\end{prop}

\begin{proof}
From \eqref{e2.3}, it
 can be proved that $$\int_0^\infty \Big(a_2-c_2(\phi+\nu)-\dfrac{\beta_2^2}2(\phi+\nu)^2\Big)\pi_1^*(d\phi)\geq \dfrac{4\lambda_1}5$$ for sufficiently small $\nu$.
Let $\hat M$ be as in Lemma \ref{lm3.4}.
By the ergodicity of $\varphi(t)$ (see \eqref{erg}),
there is $T=T(\eps, H)>\dfrac{25\hat M^2H}{\eps^2\lambda_1^2}$ such that
$$\PP\left\{\dfrac1T\int_0^T \Big(a_2-c_2(\varphi_{H}(t)+\nu)-\dfrac{\beta_2^2}2(\varphi_{H}(t)+\nu)^2\Big)dt\geq \dfrac{3\lambda_1}5\right\}\geq 1-\eps.$$
By the uniqueness of solution, $\varphi_{x }(t)\leq\varphi_{H}(t)$ a.s. for all $x \in[H^{-1},H]$. As a result, $\PP(\Omega^{z}_2)\geq 1-\eps$ where
$$\Omega^{z}_2=\left\{\int_0^T \Big(a_2-c_2(\varphi_{x }(t)+\nu)-\dfrac{\beta_2^2}2(\varphi_{x }(t)+\nu)^2\Big)dt\geq \dfrac{3\lambda_1}5T\right\}\geq 1-\eps.$$
In view of Lemma \ref{lm3.3},
we can choose $\sigma=\sigma(\eps, H)>0$ such that $b_1\sigma+\dfrac{\alpha_2^2}2\sigma^2<\dfrac{\lambda}5$ and
 $$\PP(\Omega^{z }_3)\geq 1-\eps \text{ where } \Omega^{z }_3=\{|\varphi_{x }(t)-X_{z }(t)|<\nu\,\forall 0\leq t\leq T\wedge\tau_{z }^\sigma \}.$$
By virtue of Lemma \ref{lm3.2}, there is a $\delta_0=\delta_0(\eps, H)$ satisfying that for all $z \in[H^{-1},H]\times(0,\delta_0]$,
$$\PP(\Omega^{z }_4)\geq 1-\eps \text{ where }\Omega^{z }_4=\{\tau_{z }^\sigma \geq T\}.$$
Since $T>\dfrac{25\hat M^2H}{\eps^2\lambda_1^2}$, it follows from Lemma \ref{lm3.4} that
$$\PP(\Omega^{z }_5)\geq 1-\eps \text{ where } \Omega^{z }_5=\left\{\left|\int_0^T\alpha_2 Y_{z }(t)dB_3(t)+\beta_2 X_{z }(t)dB_2(t)\right|\leq \dfrac{\lambda_1}5T\right\}.$$
For $z \in [H^{-1},H]\times(0,\delta_0]$ and $\omega\in\hat\Omega^{z }=\cap_{i=2}^5\Omega^{z }_i$ we have
$$
\begin{aligned}
\ln Y_{z } (T)-\ln y \geq& \int_0^T \Big(a_2-c_2X_{z }(t)-\dfrac{\beta_2}2X^2_{z }(t)\Big)dt-b_2\int_0^TY_{z }(t)dt-\dfrac{\alpha^2_2}2\int_0^TY_z^2(t)dt\\
&-\left|\int_0^T\alpha_2 Y_{z }(t)dB_3(t)+\beta_2 X_{z }(t)dB_2(t)\right|\\
\geq&\int_0^T \Big(a_2-c_2[\varphi_{x }(t)+\nu]-\dfrac{\beta_2^2}2\big([\varphi_{x }(t)+\nu]\big)^2\Big)dt-\dfrac{2\lambda_1}{5}T\geq\dfrac{\lambda_1}{5}T.
\end{aligned}
$$
The proof is complete by noting that $\PP(\hat\Omega^z)=\PP\left(\cap_{i=2}^5\Omega^{z }_i\right)>1-4\eps$.
\end{proof}

\begin{prop}\label{prop3.2}
Suppose that $\lambda_1>0$. Then, for any $\Delta>0$, there are $T=T(\Delta)$ and $\delta_2=\delta_2(\Delta)>0$ such that
$$\limsup\limits_{n\to\infty}\dfrac1n\sum_{k=0}^{n-1}\PP\{Y_{z_0}(kT)\leq\delta_2\}\leq\Delta,\,\forall z_0\in\R^{2,\circ}_+$$
\end{prop}

\begin{proof}
Let $\eps=\eps(\Delta)\in(0,1)$ and $H=H(\Delta)>1$  be chosen later.
Put $\Lambda={\theta_2(1+2|H|)^{1/p^*}}\eps^{-1}$ where $\theta_2$ is as in \eqref{e3.1}.
As a result of \eqref{e3.3} and Proposition \ref{prop3.1}, there are $\delta_0\in(0,1)$ and $T>1$ such that $\PP\{\hat\Omega_0^{z}\}>1-5\eps\,\forall z\in[H^{-1},H]\times(0,\delta_0]$ where
$\hat\Omega_0^{z}=\big\{\dfrac{\lambda_1}5T\leq\ln Y_{z }(T)-\ln y\leq\Lambda T\big\}.$
Let $L_1=\Lambda T+\ln \dfrac{H}{\delta_0}$.
Since $|\ln H-\ln(Y_z (T))|\leq |\ln H-\ln y|+\big|\ln Y_{z }(T)-\ln y \big|,$
it follows from \eqref{e3.3} that if $z\in [H^{-1},H]\times(\delta_0, H]$,
\begin{equation}\label{e3.11a}
\PP\{|\ln H-\ln(Y_z (T))|\geq L_1\}\leq \PP\{\big|\ln Y_{z }(T)-\ln y \big|\geq \Lambda T\}\leq5\eps.
\end{equation}
Let $\delta_1,\delta_2$ satisfy $L_1=\ln H-\ln \delta_1$ and $L_2:=L_1+\Lambda T=\ln H-\ln \delta_2.$
Note that $\delta_2<\delta_1<\delta_0.$
Define $U(y)=\big(\ln H-\ln y\big)\vee L_1.$
Clearly,
\begin{equation}\label{e3.11}
U(y_1)-U(y_2)\leq |\ln(y_1)-\ln(y_2)|.
\end{equation}
We now estimate $\dfrac1T\big(\E U(Y_{z}(T))-U(y)\big)$ for different
$z$.
First, for any $z\in\R^{2,\circ}_+$, using \eqref{e3.1} and \eqref{e3.11} we have for
$\omega\in\hat\Omega_0^{z,c}=\Omega\setminus\hat\Omega_0^z$ that
\begin{equation}\label{e3.11b}
\dfrac1T\E\1_{\hat\Omega_0^{z,c}}\big(U(Y_{z}(T))-U(y)\big)\leq \dfrac1T\E\1_{\hat\Omega_0^{z,c}}\big|\ln Y_{z}(T)-\ln y\big|\leq
\theta_2(1+2H)^{1/p^*}\big(\PP(\hat\Omega_0^{z,c})\big)^{1/q^*}.
\end{equation}

If $z\in D_3:=[H^{-1},H]\times(0,\delta_2]$, $U(y)=\ln H-\ln y\geq L_1+\Lambda T.$
In $\hat\Omega^z_0$, we have $$L_1\leq \ln H-\ln y-\Lambda T\leq \ln H-\ln (Y_{z}(T))\leq \ln H-\ln y-\dfrac{\lambda_1}{5}T.$$
As a result, if $\omega\in\hat\Omega_0^z$,
\begin{equation}\label{e3.11a-a}
\dfrac1T\big(U(Y_{z}(T))-U(y)\big)=\dfrac1T\big[\big(\ln H-\ln (Y_{z}(T))\big)-(\ln H-\ln y)\big]\leq -\dfrac{\lambda_1}{5}.
\end{equation}
Combining  \eqref{e3.11b} and \eqref{e3.11a-a} yields
\begin{equation}\label{e3.12}
\dfrac1T\big(\E U(Y_{z}(T))-U(y)\big)\leq -(1-\eps)\dfrac{\lambda_1}{5}+\theta_2(1+2H)^{1/p^*}\big(5\eps\big)^{1/q^*}\,\forall\,z\in D_3.
\end{equation}

If $z\in D_2:=[H^{-1},H]\times(\delta_2,\delta_1]$, we have $U(y)=\ln H-\ln y\geq L_1$.
If $\omega\in\hat\Omega_0^z$, $\ln H-\ln (Y_{z}(T))\leq \ln H-\ln y=U(y)$.
As a result, $U(Y_{z}(T))\leq U(y)$ in $\hat\Omega_0^z.$
This and \eqref{e3.11b} imply
\begin{equation}\label{e3.13}
\dfrac1T\big(\E U(Y_{z}(T))-U(y)\big)\leq \theta_2(1+2H)^{1/p^*}\big(5\eps\big)^{1/q^*}\,\forall\,z\in D_2.
\end{equation}

If $z\in D_1:=[H^{-1},H]\times(\delta_1,\delta_0]$
and $\omega\in\hat\Omega_z$, we have $\ln H-\ln (Y_{z}(T))\leq \ln H-\ln y\leq L_1$, which implies that $U(Y_{z}(T))=L_1=U(y).$
Consequently, we also have
\begin{equation}\label{e3.14}
\dfrac1T\big(\E U(Y_{z}(T))-U(y)\big)\leq \theta_2(1+2H)^{1/p^*}\big(5\eps\big)^{1/q^*}\,\forall\,z\in D_1.
\end{equation}

If $z\in D_0:=[H^{-1},H]\times(\delta_0,H]$,
we have $U(z)=L_1$.
In view of \eqref{e3.11a}, \eqref{e3.11}, and \eqref{e3.1} we have
\begin{equation}\label{e3.15}
\dfrac1T\big(\E U(Y_{z}(T))-U(y)\big)\leq \theta_2(1+2H)^{1/p^*}\big(5\eps\big)^{1/q^*}\,\forall\,z\in D_0.
\end{equation}

For any $z_0\in\R^{2,\circ}_+$, it follows from the Markov property of $Z(t)$ that
$$
\dfrac{1}{T}\E\big(U(Y_{z_0}(kT+T))-U(Y_{z_0}(kT)\big)\leq \int_{\R^{2,\circ}_+}\PP\{Z_{z_0}(kT)\in dz\}\left[\dfrac1T\big(\E U(Y_{z}(T))-U(y)\big)\right].
$$
Subsequently, letting $D:=[H^{-1},H]\times(0,H]=\cup_{i=0}^3D_i$ and using \eqref{e3.1}, \eqref{e3.12}, \eqref{e3.13}, \eqref{e3.14}, and \eqref{e3.15} we have
\begin{equation}\label{e3.15a}
\begin{aligned}
\dfrac{1}{T}\E\big(&U(Y_{z_0}(kT+T))-U(Y_{z_0}(kT)\big)\\
&\leq \big(-(1-\eps)\dfrac{\lambda_1}{5}+\eps_1\big)\PP\{Z_{z_0}(kT)\in D_3\}\\
&\quad+\eps_1\PP\{Z_{z_0}(kT)\in D\setminus D_3\}+\E\1_{\{Z_{z_0}(kT)\notin D\}}(1+|Z_{z_0}(kT)|)^{1/p^*}\\
&\leq -(1-\eps)\dfrac{\lambda_1}{5}\PP\{Z_{z_0}(kT)\in D_3\}+\big(\PP\{Z_{z_0}(kT)\notin D\}\big)^{1/q^*}(1+\E|Z_{z_0}(kT)|)^{1/p^*}+\eps_1,
\end{aligned}
\end{equation}
where $\eps_1=\theta_2(1+2H)^{1/p^*}\big(5\eps\big)^{1/q^*}.$
In view of  Proposition \ref{prop2.1} part (i),
\begin{equation}\label{e3.16}\barray
\ad \disp
\limsup\limits_{k\to\infty}\big(\PP\{Z_{z_0}(kT)\notin D\}\big)^{1/q^*}(1+\E|Z_{z_0}(kT)|)^{1/p^*}\\
\aad \ \leq \limsup\limits_{k\to\infty}\dfrac{1+2\E V(Z_{z_0}(kT))}{H^{1/q^*}}\leq\dfrac{1+2M_0}{H^{1/q^*}}
, \earray\end{equation}
and
\begin{equation}\label{e3.16b}\limsup\limits_{k\to\infty}\PP\{Z_{z_0}(kT)\notin D\}\leq \limsup\limits_{k\to\infty}\dfrac{\E V(Z_{z_0}(kT))}{H}\leq \dfrac{M_0}{H}.
\end{equation}
Clearly,
\begin{equation}\label{e3.16a}
\liminf\limits_{n\to\infty}\dfrac{1}{nT}\sum_{k=0}^{n-1}\E\big[U(Y_{z_0}(kT+T))-U(Y_{z_0}(kT)\big]=\liminf\limits_{n\to\infty}\dfrac{\E U(Y_{z_0}(nT))}{nT}\geq0.
\end{equation}
We derive from \eqref{e3.15a}, \eqref{e3.16}, and \eqref{e3.16a} that
\begin{equation}\label{e3.17}
\begin{aligned}
\limsup\limits_{n\to\infty}\dfrac{1}{n}\sum_{k=0}^{n-1}\PP\{Z_{z_0}(kT)\in D_3\}\leq \dfrac{5}{(1-\eps)\lambda_1}\left[\eps_1+\dfrac{1+2M_0}{H^{1/q^*}}\right].
\end{aligned}
\end{equation}
Note that $\PP\{Y_{z_0}(kT)\leq\delta_2\}\leq \PP\{Z_{z_0}(kT)\in D_3\}+\PP\{Z_{z_0}(kT)\notin D\}.$
In view of \eqref{e3.16b} and \eqref{e3.17},
by choosing $H=H(\Delta)$ sufficiently large and then choosing $\eps=\eps(\Delta)$ sufficiently small, we
obtain the desired result.
\end{proof}

\begin{proof}[Proof of Theorem \ref{thm2.1}]
Let any $\eps>0$. Since $\lambda_2>0$,
similar to Proposition \ref{prop3.2},
there exist  $T'>1$ and $\delta_2'>0$ such that
$$\limsup\limits_{n\to\infty}\dfrac1n\sum_{k=0}^{n-1}\PP\{X_{z_0}(kT')\leq\delta_2'\}\leq\Delta,\,z_0\in\R^{2,\circ}_+.$$
Moreover, it can be seen in the proof of Proposition \ref{prop3.2} that we can choose any sufficiently large $T'$ and sufficiently small $\delta_2'$ satisfying the above estimate. As a result, without loss of generality, we can choose $T'=T$ and $\delta_2'=\delta_2$.
Consequently,
$$\limsup\limits_{n\to\infty}\dfrac1n\sum_{k=0}^{n-1}\PP\{|X_{z_0}(kT)|\wedge|Y_{z_0}(kT)|\leq\delta_2\}\leq2\Delta,\,z_0\in\R^{2,\circ}_+.$$
This
together with part (i) of Proposition \ref{prop2.1}  implies that
there is a compact set $G\subset\R^{2,\circ}_+$ such that
$$\liminf\limits_{n\to\infty}\dfrac1n\sum_{k=0}^{n-1}\PP\{Z_{z_0}(kT)\in G\}\geq 1-3\Delta,\,z_0\in\R^{2,\circ}_+.$$
Thanks to (ii) of Proposition \ref{prop2.1}, there is an $\ell>1$ such that $\PP\{\ell^{-1}\leq X_{z}(t),Y_{z}(t)\leq\ell\}\geq 1-\Delta$ for all $z\in G, t\leq T$.
By the Markov property,
$$\PP\{\ell^{-1}\leq X_{z_0}(kT+t),Y_{z_0}(kT+t)\leq\ell\}\geq (1-\Delta)\PP\{Z_{z_0}(kT)\in G\}\,\forall\, t\leq T.$$
Thus, for any $z_0\in\R^{2,\circ}_+$,
$$\liminf\limits_{n\to\infty}\dfrac1{nT}\int_0^{nT}\PP\{\ell^{-1}\leq X_{z_0}(t),Y_{z_0}(t)\leq\ell\}dt\geq (1-3\Delta)(1-\Delta)\geq 1-4\Delta.$$
It
implies that
$$\liminf\limits_{t\to\infty}\dfrac1{t}\int_0^{t}\PP\{\ell^{-1}\leq X_{z_0}(s),Y_{z_0}(s)\leq\ell\}ds\geq 1-4\Delta,$$
which implies the existence of an invariant probability measure.
The rest of the
results of Theorem \ref{thm2.1} therefore follows from the non-degeneracy of the diffusion; see \cite{LB} or \cite{RK}.
\end{proof}

\section{Competitive Exclusion}\label{sec:com}
To prove Theorem \ref{thm2.1}  (the coexistence), we need only estimate the behavior of the solution near the
boundary for a sufficiently long but finite time.
In contrast, to prove Theorems \ref{thm2.2} and \ref{thm2.3}, we have to estimate the difference $\varphi_{x }(t)-X_{z }(t)$ in an infinite interval.
Note that in the deterministic case, the inverse $x^{-1}(t)$ of the solution to a logistic equation
$${d x(t) \over dt}= x(t)(a_1-b_1x(t))$$
satisfies a linear differential equation which is much easier to work with. Motivated by this, we
consider the difference $\varphi_{x }^{-1}(t)-X_{z }^{-1}(t).$

\begin{lem}\label{lm4.1}
For any $H, T>1,\eps>0, \gamma$, there is $\tilde\sigma>0$ such that $\forall z \in [H^{-1},H]\times(0,\tilde\sigma]$
$$\PP\left\{\left|\dfrac1{\varphi_{x }(t)}-\dfrac1{X_{z }(t)}\right|<\gamma\,\forall 0\leq t\leq T\wedge\tau_{z }^{\tilde\sigma}\right\}\geq 1-\eps.$$
\end{lem}

\begin{proof}
In view of (ii) of Proposition \ref{prop2.1}, we can find $\hat H=\hat H(\eps,H, T)>1$ such that for all $z \in [H^{-1},H]\times[0,H]$
$$\PP\{\hat H^{-1}\leq X_{z }(t), \varphi_z(t)\leq \hat H\,\forall t\leq T\}\geq 1-\dfrac{\eps}2.$$
When $\hat H^{-1}\leq X_{z }(t), \varphi_z(t)\leq \hat H$, we have $|\varphi_{x }^{-1}(t)-X_{z }^{-1}(t)|\leq \hat H^2|\varphi_x(t)-X_z(t)|$.
Applying Lemma \ref{lm3.3}, we obtain the desired result.
\end{proof}

\begin{prop}\label{prop4.1}
Assume that $\lambda_1<0$. For any $H>1, \eps, \gamma>0, \lambda\in (0,-\lambda_1)$, there is a $\tilde\delta>0$ such that
$$\PP\left(\left\{\limsup\limits_{t\to\infty}\dfrac{\ln Y_{z}(t)}t\leq-\lambda\right\}\cap\left\{\left|\dfrac1{\varphi_{x }(t)}-\dfrac1{X_{z }(t)}\right|\leq\gamma\,\forall t\geq0\right\}\right)\geq 1-5\eps\,\forall\,z\in[H^{-1},H]\times[0,\tilde\delta].$$
\end{prop}

\begin{proof}
Consider the case
$$\lambda_1=a_2-\int_0^\infty(c_2\phi+\dfrac{\beta_2^2}2\phi^2)f^*_1(\phi)d\phi<0.$$
Let any $\lambda\in(0, -\lambda_1)$ and $d=\dfrac{-\lambda_1-\lambda}4.$
Since $$\int_0^\infty\left(c_2\phi+\dfrac{\beta_2^2}2\phi^2\right)f^*_1(\phi)d\phi=a_2-\lambda_1<\infty,$$ we can find $\eta_1,\eta_2,\eta_3\in(0,1)$ such that
$$\int_{\eta_1}^{\infty}\left(c_2(\phi-\eta_1)+
\dfrac{\beta_2^2(1-\eta_3)}2\big(\phi-\eta_1)^2\right)f^*_1(\phi)d\phi\geq a_2-\lambda_1-\vartheta=a_2+\lambda-3d$$
and
$$\int_{\eta_2^{-1}}^{\infty}\left(c_2(\phi-\eta_1)+
\dfrac{\beta_2^2(1-\eta_3)}2\big(\phi-\eta_1)^2\right)f^*_1(\phi)d\phi\leq d.$$
By the ergodicity \eqref{erg},
there is a $T_1=T_1(\eps, H)$ such that
with a probability greater than $1-\eps$, we have
$$\dfrac1t\int_0^t\1_{\{\eta_1\leq\varphi_{H^{-1}}(s)\}}\Big(c_2\big(\varphi_{H^{-1}}(s)
-\eta_1\big)+\dfrac{\beta_2^2(1-\eta_3)}2\big(\varphi_{H^{-1}}(s)-\eta_1\big)^2\Big)ds\geq a_2+\lambda-2d\,\forall t\geq T_1,$$
and
$$\dfrac1t\int_0^t\1_{\{\eta_2^{-1}\leq\varphi_{H}(s)\}}\Big(c_2\big(\varphi_{H}(s)
-\eta_1\big)+\dfrac{\beta_2^2(1-\eta_3)}2\big(\varphi_{H}(s)-\eta_1\big)^2\Big)ds\leq 2d\,\forall t\geq T_1,$$

	Combining this with the fact that $\varphi_{H^{-1}}(s)\leq\varphi_{x}(s)\leq\varphi_{H}(s)$ a.s. $\forall s\geq0$, $\forall x>H^{-1}$,  with a probability greater than $1-\eps$ we have
$$\dfrac1t\int_0^t\1_{\{\eta_1\leq\varphi_{x}(s)\}}\Big(c_2\big(\varphi_{x}(s)
-\eta_1\big)+\dfrac{\beta_2^2(1-\eta_3)}2\big(\varphi_{x}(s)-\eta_1\big)^2\Big)ds\geq a_2+\lambda-2d\,\forall t\geq T_1,$$
and
$$\dfrac1t\int_0^t\1_{\{\eta_2^{-1}\leq\varphi_{x}(s)\}}\Big(c_2\big(\varphi_{x}(s)
-\eta_1\big)+\dfrac{\beta_2^2(1-\eta_3)}2\big(\varphi_{x}(s)-\eta_1\big)^2\Big)ds\leq 2d\,\forall t\geq T_1.$$
It follows that	$\PP(\Omega_6^z)\geq1-\eps$, where
$$
\Omega_6^z=\Big\{a_2-\dfrac1t\int_0^t\1_{\{\eta_1\leq\varphi_{x}(s)\leq\eta_2^{-1}\}}\Big(c_2\big(
\varphi_{x}(s)-\eta_1\big)+\dfrac{\beta_2^2(1-\eta_3)}2\big(\varphi_{x}(s)-\eta_1\big)^2\Big)ds\leq-\lambda\,\forall t\geq T_1\Big\}.$$
Observe that the estimate $x\geq\big(\phi-\eta_1)\1_{\{\eta_1\leq\phi\leq\eta_2^{-1}\}}$ holds if $|\phi^{-1}-x^{-1}|\leq \eta_1\eta_2^2.$
Indeed,
 if $x\geq\phi$ or $\phi>\eta_2^{-1}$ or $\phi<\eta_1$, we obviously have
 $x\geq \big(\phi-\eta_1)\1_{\{\eta_1\leq\phi\leq\eta_2^{-1}\}}.$
In the case when $x<\phi\leq\eta_2^{-1}$ and $\phi>\eta_1$ and $|\phi^{-1}-x^{-1}|\leq \eta_1\eta_2^2$, we have $\phi-x\leq \eta_2^{-2}|\phi^{-1}-x^{-1}|\leq\eta_1$, which implies that
$x\geq \phi-\eta_1= \big(\phi-\eta_1)\1_{\{\eta_1\leq\phi\leq\eta_2^{-1}\}}.$
Consequently, if $\omega\in \Omega_6^{z}\cap\{\vartheta_z\geq T_1\}$,
we have
\begin{equation}\label{e4.1}
\dfrac1t\int_0^t\Big(a_2-c_2X_{z}(s)-\dfrac{\beta_2^2(1-\eta_3)}2X^2_{z}(s)\Big)ds\leq -\lambda\,\forall t\in [T_1,\vartheta_z]
\end{equation}
where $\vartheta_z=\inf\left\{t>0: \left|\dfrac1{\varphi_{x }(t)}-\dfrac1{X_{z }(t)}\right|\geq \gamma_0:=\gamma\wedge(\eta_1\eta_2^2)\right\}.$
Recall that
\begin{equation}\label{e4.2}
\begin{aligned}
\ln Y_{z }(t)=\ln y &+\int_0^{t}\Big[a_2-b_2Y_z(s)-c_2X_{z }(s)-\dfrac{\alpha^2_2}2 Y_{z }(s)-\dfrac{\beta_2^2}2X_{z }^2(s)\Big]ds\\
&+\int_0^{t}\big[\alpha_2Y_{z }(s)dB_3(s)+\beta_2X_{z }(s)dB_2(s)\big].
\end{aligned}
\end{equation}
Setting $$
\begin{aligned}
\Omega_7^{z}=\bigg\{\int_0^{t}&\big[\alpha_2Y_{z }(s)dB_3(s)+\beta_2X_{z }(s)dB_2(s)\big]\\&\leq\dfrac1{\eta_3}\ln\dfrac1\eps+ \dfrac{\eta_3}2\int_0^{t}\big[\alpha_2^2Y^2_{z }(s)+\beta_2^2X^2_{z }(s)\big]ds\,\forall\,t\geq0\bigg\},
\end{aligned}
$$
it follows from \eqref{e4.1} and \eqref{e4.2} that for $\omega\in \Omega_6^{z}\cap\Omega_7^{z}\cap\{\vartheta_z\geq T_1\}$,
we have
\begin{equation}\label{e4.3}
\ln Y_{z}(t)\leq \ln y+\dfrac1{\eta_3}\ln\dfrac1\eps -\lambda t\,\forall t\in [T_1,\vartheta_z].
\end{equation}
If $y\leq 1$, putting $\tilde m_1=\exp\left(\dfrac1{\eta_3}\ln\dfrac1\eps\right)=\dfrac{\exp(\eta_3^{-1})}\eps$, we have
\begin{equation}\label{e4.4}
Y_{z}(t)\leq \tilde m_1\exp(-\lambda t)\forall t\in [T_1,\vartheta_z] \text{ if } \omega\in \Omega_6^{z}\cap\Omega_7^{z}\cap\{\vartheta_z\geq T_1\}.
\end{equation}
Now, we estimate $\dfrac1{\varphi_{x }(t)}-\dfrac1{X_{z }(t)}$ for a
larger time interval.
It follows from It\^o's formula that
$$d\left(\dfrac1{\varphi_{x }(t)}-\dfrac1{X_{z }(t)}\right)^2=f(\varphi_{z }(t),X_{z }(t), Y_{z }(t))dt+g(\varphi_{z }(t),X_{z }(t), Y_{z }(t))dB_2(t),$$
where
\begin{equation}\label{e4.5}
\begin{aligned}
f(\phi, x, y)=&-2a_1\left(\dfrac1\phi-\dfrac1x\right)^2+2\alpha_1^2(\phi-x)\left(\dfrac1\phi-\dfrac1x\right)-2(c_1y+\beta_1^2y^2)\dfrac1x\left(\dfrac1\phi-\dfrac1x\right)+\beta_1^2y^2\dfrac1{x^2}\\
= &-2a_1\left(\dfrac1\phi-\dfrac1x\right)^2-2(c_1y+\beta_1^2y^2)\left[\dfrac1\phi\left(\dfrac1\phi-\dfrac1x\right)-\left(\dfrac1\phi-\dfrac1x\right)^2\right]\\
&+\beta_1^2y^2\left(\dfrac1\phi-\left(\dfrac1\phi-\dfrac1x\right)\right)^2\\
\leq &-(2a_1-2c_1y-5\beta_1^2y^2)(\dfrac1\phi-\dfrac1x)^2+\beta_1^2y^2(\dfrac1\phi)^2-2c_1y\dfrac1\phi\left(\dfrac1\phi-\dfrac1x\right)\\
\leq&-(a_1-2c_1y-5\beta_1^2y^2)\left(\dfrac1\phi-\dfrac1x\right)^2+\big(\beta_1^2+\dfrac{c_1^2}{a_1}\big)y^2\dfrac1{\phi^2},
\end{aligned}
\end{equation}
and
\begin{equation}\label{e4.6}
g(\phi, x, y)=2\beta_1y\dfrac1x\left(\dfrac1\phi-\dfrac1x\right)=2\beta_1y\left(\dfrac1\phi-\dfrac1x\right)\left[\dfrac1\phi-\left(\dfrac1\phi-\dfrac1x\right)\right].
\end{equation}
Putting
$$
\begin{aligned}
\Omega_8^{z}=\left\{\int_0^{t}g\big(\varphi_{x }(s), X_{z }(s), Y_{z }(s)\big)dB_2(s)\leq \dfrac{\gamma_0^2}2+ \tilde m_2\int_0^{t}\big[g\big(\varphi_{x }(s), X_{z }(s), Y_{z }(s)\big)\big]^2ds\,\forall\,t\geq0\right\},
\end{aligned}
$$
where $\tilde m_2=\dfrac1{\gamma_0^2}\ln\dfrac1\eps.$
For
$\omega\in \Omega_8^{z}$,
\begin{equation}\label{e4.7}
\begin{aligned}
\left(\dfrac1{\varphi_{x }(t)}-\dfrac1{X_{z}(t)}\right)^2\leq\dfrac{\gamma^2_0}2+\int_0^{t}\Big(f(\varphi_{z }(s),X_{z }(s), Y_{z }(s))+\tilde m_2 g^2(\varphi_{z }(s),X_{z }(s), Y_{z }(s))\Big)ds.
\end{aligned}
\end{equation}
We deduce from \eqref{e4.5} and \eqref{e4.6} that
\begin{equation}\label{e4.8}
f(\phi, x, y)+\tilde m_2 g^2(\phi, x, y)\leq \tilde m_3y^2\dfrac1{\phi^2}\text{ if }\left|\dfrac1\phi-\dfrac1x\right|\leq 1 \text{ and }(8\tilde m_2+5)\beta_1^2y^2+2c_1y\leq a_1,
\end{equation}
where $\tilde m_3=8\tilde m_2\beta_1^2+\beta_1^2+\dfrac{c_1^2}{a_1}$.
In view of \eqref{e2.2}, there is a $T_2=T_2(\eps,H)>0$ such that $$\PP\left\{\dfrac1t\int_0^t\dfrac1{\varphi_{H^{-1}}^2(s)}ds\leq 2Q_{-2}\,\forall t\geq T_2\right\}\geq 1-\eps.$$
As a result, for all $x\in[H^{-1},H]$,
\begin{equation}\label{e4.11}
\PP\{\Omega_9^{z}\}\geq1-\eps \text{ where }\Omega_9^{z}=\Big\{\Upsilon_x(t):=\int_0^t\dfrac1{\varphi_{x}^2(s)}ds\leq 2Q_{-2}t\,\,\forall t\geq T_2\Big\}.
\end{equation}
Clearly, we can choose $T_3=T_3(\eps, H)\geq T_1\vee T_2$ such that
\begin{equation}\label{e4.9}
2\tilde m_1^2Q_{-2}\left(e^{-2\lambda t}t+2\lambda\int_{T_3}^t e^{-2\lambda s}sds\right)< \dfrac{\gamma_0^2}{4\tilde m_3}\,\forall t\geq T_3,
\end{equation}
and $\tilde\sigma=\tilde\sigma(\eps, H)<1$ sufficiently small such that
\begin{equation}\label{e4.10}
(8\tilde m_2+5)\beta_1^2\tilde\sigma^2+2c_1\tilde\sigma\leq a_1\text{ and } 2\tilde\sigma^2Q_{-2}T_3\leq \dfrac{\gamma_0^2}{4\tilde m_3}.
\end{equation}
In view of Lemmas \ref{lm3.2} and \ref{lm4.1},
we can find a $\tilde\delta=\tilde\delta(\eps, H)$ so small that
\begin{equation}\label{e4.12}
\ln \tilde\delta+\dfrac1{\eta_3}\ln\dfrac1\eps-\lambda T_3< \ln\tilde\sigma
\end{equation}
and
$$\PP(\Omega_{10}^{z})\geq 1-\eps\,\forall z\in [H^{-1},H]\times(0,\tilde\delta]\text{ where }\Omega_{10}^{z}=\{\zeta_{z}:=\vartheta_z\wedge\tau_z^{\tilde\sigma}\geq T_3\}.$$
It follows from \eqref{e4.7}, \eqref{e4.8}, and \eqref{e4.10} that when $\omega\in \Omega_8^{z}$ we have
\begin{equation}\label{e4.13}
\Big(\dfrac1{\varphi_{x }(t\wedge\zeta_{z})}-\dfrac1{X_{x }(t\wedge\zeta_{z})}\Big)^2\leq\dfrac{\gamma_0^2}2+\tilde m_3\int_0^{t\wedge\zeta_{z}}\dfrac{Y^2_{x }(s)}{\varphi^2_{x }(s)}ds\,\forall\, t\geq0.
\end{equation}
We have $\PP( \Omega_7^{z}), \PP( \Omega_8^{z})\geq1-\eps$ by the exponential martingale inequality. Hence $\PP(\tilde\Omega^z)\geq1-5\eps$ where
$\tilde\Omega^{z}=\cap_{i=6}^{10}\Omega^{z}_i$.
For $\omega\in \tilde\Omega^{z}$ and $t\geq T_3$, by integration by parts and using \eqref{e4.4}, \eqref{e4.9}, \eqref{e4.10}, and \eqref{e4.11}, we yield
\begin{equation}\label{e4.14}
\begin{aligned}
\int_0^{t\wedge\zeta_{z}}\dfrac{Y^2_{x }(s)}{\varphi^2_{x }(s)}ds=&\int_0^{T_3}\dfrac{Y^2_{x }(s)}{\varphi^2_{x }(s)}ds+\int_{T_3}^{t\wedge\zeta_{z}}\dfrac{Y^2_{x }(s)}{\varphi^2_{x }(s)}ds\\
\leq& \tilde\sigma^2\int_0^{T_3}\dfrac{1}{\varphi^2_{x }(s)}ds+\tilde m_1^2\int_{T_3}^{t\wedge\zeta_{z}}\exp(-2\lambda s)d\Upsilon_x(s)\\
\leq& 2\tilde\sigma^2Q_{-2}{T_3}+\tilde m_1^2\left[e^{-2(t\wedge\zeta_{z})}\Upsilon_x(t\wedge\zeta_{z})+2\lambda\int_{T_3}^{t\wedge\zeta_{z}}e^{-2\lambda s}\Upsilon_x(s)ds\right]
<\dfrac{\gamma_0^2}{2\tilde m_3}
\end{aligned}
\end{equation}
It follows from \eqref{e4.13} and \eqref{e4.14} that if $\omega\in \tilde\Omega^{z}$,
then $$\left(\dfrac1{\varphi_{x }(t\wedge\zeta_{z})}-\dfrac1{X_{x }(t\wedge\zeta_{z})}\right)^2<\gamma_0^2.$$
As a result, in $\tilde\Omega^{z}$, $t\wedge\zeta_{z}<\vartheta_{z}\,\forall t\geq T_3$,
which implies that $\{\zeta_{z}\leq\vartheta_z\}\supset \tilde\Omega^{z}$.
Since $\zeta_z=\vartheta_z\wedge\tau_z^{\tilde\sigma}$,
we obtain $\{\tau_{z}^{\tilde\sigma}\leq\vartheta_z\}\supset \tilde\Omega^{z}$.
When $z\in[H^{-1},H]\times(0,\tilde\delta]$ and $\omega\in  \tilde\Omega^{z}$,
it follows from \eqref{e4.3} and \eqref{e4.12} that
$$\ln Y_{z}(t\wedge\tau_{z}^{\tilde\sigma})\leq \ln y+\dfrac1{\eta_3}\ln\dfrac1\eps -\lambda(t\wedge\tau_{z}^{\tilde\sigma})<\ln\tilde\sigma\,\forall t\geq T_3.$$
It means that $t\wedge\tau_{z}^{\tilde\sigma}<\tau_{z}^{\tilde\sigma}\,\forall t\geq T_3$ for any $z\in[H^{-1},H]\times(0,\tilde\delta]$ and $\omega\in  \tilde\Omega^{z}$.
Equivalently, $\tau_{z}^{\tilde\sigma}=\vartheta_z=\infty$ for $\omega\in  \tilde\Omega^{z}$ and $z\in[H^{-1},H]\times(0,\tilde\delta]$.

As a result,  for $z\in[H^{-1},H]\times(0,\tilde\delta]$
$$\PP\left\{\limsup\limits_{t\to\infty}\dfrac{\ln Y_{z}(t)}t\leq-\lambda\text { and }\left|\dfrac1{\varphi_{x }(t)}-\dfrac1{X_{z }(t)}\right|\leq\gamma_0\leq\gamma\,\forall t\geq0\right\}\geq\PP(\tilde\Omega^z)\geq 1-5\eps.$$
\end{proof}

\begin{prop}\label{prop4.2}
For any $H>1, \eps>0,\rho>0$, there is a $\bar\delta>0$ such that for all $z\in [H^{-1},H]\times (0,\bar\delta)$, we have
$$\PP\left(\left\{\limsup\limits_{t\to\infty}\left|\dfrac{1}t\int_0^tX_{z }(s)ds-Q_1\right|\leq\rho\right\}\cap\left\{\limsup\limits_{t\to\infty}\left|
\dfrac{1}t\int_0^tX^2_{z}(s)ds-Q_2\right|\leq\rho\right\}\right)\geq 1-\eps.$$
\end{prop}

\begin{proof}
Let $\bar\eta_1,\bar\eta_2,\bar\eta_3\in(0,1)$ be such that
$$\int_{\bar\eta_1}^{\bar\eta_2^{-1}}(\phi-\bar\eta_1) f^*_1(\phi)d\phi\geq Q_1-\dfrac\rho{1\vee b_1}
\text{ and }
\int_{\bar\eta_1}^{\bar\eta_2^{-1}}(\phi-\bar\eta_1)^2 f^*_1(\phi)d\phi\geq Q_2-\dfrac\rho{1\vee(\alpha_1^2/2)}.$$
In view of Proposition \ref{prop4.1}, there is a $\bar\delta>0$ such that for all $z\in [H^{-1},H]\times (0,\bar\delta)$,
$$\PP(\bar\Omega^z_1)>1-\eps\text{ where } \bar\Omega^z_1=\left\{\lim\limits_{t\to\infty}Y_z(t)=0\right\}\cap\left\{\left|\dfrac1{\varphi_{x }(t)}-\dfrac1{X_{z }(t)}\right|\leq\bar\eta_1\bar\eta_2^2\,\forall t\geq0\right\}.$$
Similar to \eqref{e4.1}, we have for $\omega\in\bar\Omega^z_1$ that
\begin{equation}\label{e4.15c}
\liminf\limits_{t\to\infty}\dfrac1t\int_0^tX_{z }(s)ds\geq  Q_1-\dfrac\rho{1\vee b_1}
\end{equation}
and
\begin{equation}\label{e4.15a}
\liminf\limits_{t\to\infty}\dfrac1t\int_0^tX_{z }^2(s)ds\geq  Q_2-\dfrac\rho{1\vee(\alpha_1^2/2)}.
\end{equation}
On the other hand, we have from It\^o's formula that
$$
\begin{aligned}
\dfrac{\ln X_z(t)}t=&\dfrac{\ln x}t+a_1-\dfrac1t\int_0^t\left(b_1 X_z(s)+\dfrac{\alpha_1^2}2 X_z^2(s)+\dfrac{\beta^2_1}2Y^2_z(s)\right)ds+c_1\int_0^tY_z(s)ds\\
&+\dfrac1t\int_0^t\big(\alpha_1 X_z(s)dB_1(s)+\beta_1Y_z(s)dB_2(s)\big)
,\end{aligned}
$$
and
$$
\begin{aligned}
\dfrac{\ln \varphi_{x}(t)}t=\dfrac{\ln x}t+a_1-\dfrac1t\int_0^t\left(b_1 \varphi_{x}(s)+\dfrac{\alpha_1^2}2 \varphi_{x}^2(s)\right)ds+\dfrac1t\int_0^t\alpha_1 \varphi_{x}(s)dB_1(s).
\end{aligned}
$$
Using the ergodicity of $\varphi_x(t)$
and the strong law of large numbers for martingales we have
$$\lim\limits_{t\to\infty}\left[\dfrac1t\int_0^t\left(b_1 \varphi_{x}(s)+\dfrac{\alpha_1^2}2 \varphi_{x}^2(s)\right)ds+\dfrac1t\int_0^t\alpha_1 \varphi_{x}(s)dB_1(s)\right]=b_1Q_1+\dfrac{\alpha^2_1}2Q_2 \quad\text{a.s.}$$
By direct calculation, $b_1Q_1+\dfrac{\alpha^2_1}2Q_2=a_1$, which implies that
 $$\lim\limits_{t\to\infty}\dfrac{\ln \varphi_{x}(t)}t=0\quad\text{a.s.}$$
 Note that, if $\omega\in \bar\Omega^z_1$,  we have $\left|\dfrac1{\varphi_{x }(t)}-\dfrac1{X_{z }(t)}\right|\leq\bar\eta_1\bar\eta_2^2$.
 Hence $$\ln X_{z}(t)=-\ln\dfrac1{X_{z }(t)}\geq-\left|\ln \dfrac1{\varphi_{x}(t)}\right|-|\ln (\eta_1\bar\eta_2^2)|.$$
 As a result,
 $$\liminf\limits_{t\to\infty}\dfrac{\ln X_{z}(t)}t\geq0\quad\text{ for almost } \omega\in \bar\Omega^z_1.$$
 Using this estimate and arguments similar the proof of \cite[Theorem 2.2]{DDT} as well as the convergence of $Y_z(t)$ to $0$ in $\bar\Omega^z_1$, we can show that
\begin{equation}\label{e4.15b}
\limsup\limits_{t\to\infty}\dfrac1t\int_0^t\big(b_1 X_z(s)+\dfrac{\alpha_1^2}2 X_z^2(s)\big)ds\leq a_1=b_1Q_1+\dfrac{\alpha^2_1}2Q_2	\text{ for almost } \omega\in\bar\Omega^z_1.
\end{equation}
The claim of the proposition is derived from \eqref{e4.15c}, \eqref{e4.15a}, and \eqref{e4.15b}.
\end{proof}

We are now in a position to prove Theorems \ref{thm2.2} and \ref{thm2.3}.

\begin{proof}[Proof of Theorem \ref{thm2.2}]
Suppose $\lambda_1<0$ and $\lambda_2>0$.
Consider
any $\eps, \gamma>0$ and $\lambda\in(0,-\lambda_1)$.
In view of Proposition \ref{prop2.1}, there is an $H>1$ such that
\begin{equation}\label{e4.16}
\limsup\limits_{t\to\infty}\PP\{(Y_{z}(t),X_{z}(t)\big)\in C\}\geq 1-\eps \text{
where } C:=\{H^{-1}\leq x\vee y\leq H\}.\end{equation}
By virtue of Proposition \ref{prop4.1}, there is $\tilde\delta_1>0$ such that
\begin{equation}\label{e4.18}
\PP\left\{\limsup\limits_{t\to\infty}\dfrac{\ln Y_{z}(t)}t\leq-\lambda\text { and }\left|\dfrac1{\varphi_{x }(t)}-\dfrac1{X_{z }(t)}\right|\leq\gamma\,\forall t\geq0
\right\}\geq 1-\eps\, \forall\, z\in C_1
,\end{equation}
where
$C_1:=[H^{-1},H]\times(0,\tilde\delta_1).$
Since $\lambda_2>0$, similar to Proposition \ref{prop3.2}, there is $T_4>1,\tilde\delta_2>0$ such that
\begin{equation}\label{e4.15}
\limsup\limits_{n\to\infty}\dfrac1{n}\sum_{i=0}^{n-1}\PP\{X_{z_0}(T_4)<\delta_2\}\leq\eps.
\end{equation}
\eqref{e4.18}
indicates that $Z(t)$ is not recurrent in $\R^{2,\circ}_+$.
Since the diffusion is non-degenerate, $Z(t)$ must be transient.
Note that $C_2:=C\setminus(C_1\cup \{(x,y): x<\tilde\delta_2\})$ is a compact subset of $\R^{2,\circ}_+.$
By the transience of $Z(t)$,
\begin{equation}\label{e4.17}
\lim\limits_{t\to\infty}\PP\{(Z_{z_0}(t)\in C_2\}=0.
\end{equation}
It follows from \eqref{e4.16}, \eqref{e4.15}, and \eqref{e4.17} that
$$\limsup\limits_{n\to\infty}\sum_{i=0}^{n-1}\PP\{(Z_{z_0}(iT_4)\in C_1\}\geq 1-2\eps.$$
It means that, there is $i_0$ such that $\PP\{(Z_{z_0}(i_0T_4)\in C_1\}\geq 1-3\eps$.
By the Markov property, we deduce from this and \eqref{e4.18} that
$$\PP\left\{\limsup\limits_{t\to\infty}\dfrac{\ln Y_{z_0}(t)}t\leq-\lambda\right\}\geq (1-\eps)(1-3\eps)\geq 1-4\eps.$$
It holds for any $\eps>0$ and $\lambda\in(0,-\lambda_1)$, so we claim that
\begin{equation}\label{e4.30}
\PP\left\{\limsup\limits_{t\to\infty}\dfrac{\ln Y_{z_0}(t)}t\leq\lambda_1\right\}=1.
\end{equation}
Likewise, using Proposition \ref{prop4.2} and the arguments above, we can show that
\begin{equation}\label{e4.18a}
\PP\left\{\lim\limits_{t\to\infty}\dfrac{1}t\int_0^tX_{z_0}(s)ds=Q_1\text{ and }\lim\limits_{t\to\infty}\dfrac{1}t\int_0^tX^2_{z_0}(s)ds=Q_2\right\}=1.
\end{equation}
Employing the strong law of large numbers for martingales,
\begin{equation}\label{e4.18b}
\PP\left\{\lim\limits_{t\to\infty}\dfrac{1}t\int_0^tX_{z_0}(s)dB_2(s)=0\text{ and }\lim\limits_{t\to\infty}\dfrac{1}t\int_0^tY_{z_0}(s)dB_3(s)=0\right\}=1.
\end{equation}
Applying \eqref{e4.30}, \eqref{e4.18a} and \eqref{e4.18b} to \eqref{e2.5}
leads to
$$\PP\left\{\lim\limits_{t\to\infty}\dfrac{\ln Y_{z_0}(t)}t=\lambda_1\right\}=1.$$
To prove the remaining part,
it suffices to show that the distribution of $X_{z_0}^{-1}(t)$
converges weakly to the measure $\check\pi_1$ on $(0,\infty)$ with $\check\pi_1(dx)=\frac{1}{x^2}f^*_1\left(\frac1x\right)$.
In light of Portmanteau's theorem, let $h(\cdot)$ be a Lipschitz function in $(0,\infty),$
we need to show that
$$\lim\limits_{t\to\infty}\E h(X_{z_0}^{-1}(t))= h^*:=\int_0^\infty \dfrac{h(\phi)}{\phi^2}f^*_1(\dfrac1\phi)d\phi\,\forall\,z_0\in\R_+^{2,\circ}.$$

Let $K_h>0$ be such that $|h(x_1)|\leq K_h$ and $|h(x_1)-h(x_2)|\leq K_h|x_1-x_2|$ for all $x_1,x_2\in(0,\infty)$.
We have the following estimate.
\begin{equation}\label{e4.19}
\begin{aligned}
\Big|\E h(X_z^{-1}(t))- h^*\Big|\leq&\Big|\E h(\varphi_x^{-1}(t))- h^*\Big|+K_h\gamma\PP\{|\varphi_x^{-1}(t)-X_z^{-1}(t)|\leq\gamma\}\\
&+ 2K_h\PP\{|\varphi_x^{-1}(t)-X_z^{-1}(t)|\geq\gamma\}.
\end{aligned}
\end{equation}
It follows from \eqref{e4.19} and the weak convergence of the distribution of $\varphi_x^{-1}(t)$ to $\check\pi_1$ (since the distribution of $\varphi_x(t)$ converges weakly to $\pi^*_1$) that
\begin{equation}\label{e4.20}
\begin{aligned}
\limsup\limits_{t\to\infty}\Big|\E h(X_z^{-1}(t))- h^*\Big|\leq& K_h\gamma\limsup\limits_{t\to\infty}\PP\{|\varphi_x^{-1}(t)-X_z^{-1}(t)|\leq\gamma\}\\
&+2K_h\limsup\limits_{t\to\infty}\PP\{|\varphi_x^{-1}(t)-X_z^{-1}(t)|\geq\gamma\}.
\end{aligned}
\end{equation}
By the Markov property,
\begin{equation}\label{e4.22}
\begin{aligned}
\Big|\E h(X_{z_0}^{-1}(t+i_0T_4))- h^*\Big|\leq& \int_{\R^{2,\circ}_+}\Big|\E h(X_{z}^{-1}(t))- h^*\Big|\PP\{X_{z_0}(i_0T_4)\in dz\}\\
\leq&\int_{C_1}\Big|\E h(X_{z}^{-1}(t))- h^*\Big|\PP\{X_{z_0}(i_0T_4)\in dz\}+ 2K_h\PP\{X_{z_0}(i_0T_4)\notin C_1\}.
\end{aligned}
\end{equation}
Using \eqref{e4.18} and \eqref{e4.20}, and applying Fatou's lemma to \eqref{e4.22}, we obtain
$$\limsup\limits_{t\to\infty}\left|\E h(X_z^{-1}(t+i_0T_4))- h^*\right|\leq (K_h\gamma+K_h\eps)+ 6K_h\eps.$$
It holds for any $\eps,\gamma>0$, we obtain the convergence of $\E h(X_{z_0})(t)$ to $h^*$.
The proof is complete.
\end{proof}

\begin{proof}[
Proof of Theorem \ref{thm2.3}]
For any $\eps>0$. Let $H>1$ such that
$$
\limsup\PP\{(Y_{z}(t),X_{z}(t)\big)\in C\}\geq 1-\eps \text{
where } C:=\{H^{-1}\leq x\vee y\leq H).$$
Since $\lambda_1,\lambda_2<0$,
Let $\lambda_1'\in(0,-\lambda_2)$ and $\lambda_2'\in(0,-\lambda_2)$
there is $\tilde\delta_3>0$ such that
$$\PP\left\{\lim\limits_{t\to\infty}Y_z(t) =0\right\}\geq 1-\eps\,\forall z\in C_3:=[H^{-1},H]\times(0,\tilde\delta_3)$$
and
$$\PP\left\{\lim\limits_{t\to\infty}X_{z}(t)=0\right\}\geq 1-\eps\,\forall z\in C_4:=(0,\tilde\delta_3)\times[H^{-1},H].$$

Since the diffusion is non-degenerate, for $t>0$, $\PP\{Z_{z_0}(t)\in C_3\}$
and $\PP\{Z_{z_0}(t)\in C_4\}$ are both positive.
By the Markov property, $p_{z_0}:=\PP\{\lim\limits_{t\to\infty}X_{z_0}(t)=0\}>0$ and
 $q_{z_0}:=\PP\{\lim\limits_{t\to\infty}Y_{z_0}(t)=0\}>0.$
We now show that $p_{z_0}+ q_{z_0}=1$.
Since $\limsup\limits_{t\to\infty}\PP\{Z_{z_0}(t)\in C\setminus(C_3\cup C_4)=0$,
similar to the proof of Theorem \ref{thm2.3}, there is a $T_{z_0}(\eps)>0$ such that
$$\PP\{Z_{z_0}(T_{z_0}(\eps))\in C_3\cup C_4\}\geq 1-3\eps.$$
As a consequence of the Markov property,
$$\PP\left\{\lim\limits_{t\to\infty}X_{z_0}(t)=0 \text{ or } \lim\limits_{t\to\infty}Y_{z_0}(t)=0\right\}\geq 1-4\eps.$$
Since $\eps$ is taken arbitrarily, we claim  $p_{z_0}+ q_{z_0}=1$.
Analogous to Theorem \ref{thm2.2}, we can show that
$$\PP\left\{\lim\limits_{t\to\infty}\dfrac{\ln Y_{z_0}(t)}t=\lambda_1\right\}=q_{z_0}
\
\hbox{ and }\
\PP\left\{\lim\limits_{t\to\infty}\dfrac{\ln X_{z_0}(t)}t=\lambda_1\right\}=p_{z_0}.$$
The remaining assertion can
be proved by arguments similar to
that of Theorem \ref{thm2.2}.
\end{proof}

\section{A Piecewise Deterministic Model of Competitive Type}\label{sec:pie}
In \cite{DDY1} and \cite{DD}, we considered a Kolmogorov system of competitive type under telegraph noise given by
\begin{equation}\label{e5.1}
\left\{\begin{array}{ccc}\dot x(t)&=&x(t) a({\xi(t)}, x(t),
y(t))\\
\dot y(t) &=&y(t)b({\xi(t)}, x(t), y(t)),\end{array}\right.
\end{equation}
where $\{\xi(t): {t\geq 0} \}$ be an
$\F_t$-adapted continuous-time Markov chain whose state space is
a two-element set $\M  = \{1,2\}$ and
$a_i(x, y)$ and $b_i(x, y)$ are real-valued functions defined for $ i\in \M$ and $(x,y) \in \R^2_+$,
and are continuously differentiable in $(x,y)\in
\R_+^2=\{(x,y):x\geq 0,y\geq 0\}$.
We also assume that the generator of $\xi(t)$ is given by
$Q =\left( \begin{array}{ll}  -\alpha & \alpha \\
 \beta  & -\beta \\ \end{array} \right)$
 with $\alpha > 0$ and  $\beta >0$.
Note that in the above and henceforth,
we write
$a_i(x,y)$ instead of $a(i,x,y)$ to distinguish the discrete state $i$ with the
continuous state $(x,y)$.
Due to the  telegraph noise $\xi(t)$,
the system  switches randomly
between two deterministic Kolmogorov systems
\begin{equation}
\label{e5.2}
\left\{
\begin{array}{ccc}
\dot x(t)&=&x (t) a_1( x(t), y(t))\\
\dot y(t) &=&y(t) b_1( x(t), y(t)),
\end{array}\right.
\end{equation}
\begin{equation}\label{e5.3}
\left\{\begin{array}{ccc}\dot x(t) &=&x(t) a_2( x(t), y(t))\\
\dot y(t) &=&y
(t) b_2(
x(t), y(t)).\end{array}\right.
\end{equation}
The two following assumption are imposed throughout this section.
\begin{asp}\label{asp5.1} {\rm For each $i\in \M$,
$a_i( x, y)$ and $b_i(x, y)$ are continuously differentiable in $(x,
y)\in\R^2_+.$ Moreover,
\begin{enumerate}
\item
 $\dfrac{\partial a_i(x, 0)}{\partial x}<0\,\forall\,x>0$ and $i\in \M$;\quad $\dfrac{\partial b_i( 0, y)}{\partial y}<0\,\forall\,y>0$ and $i\in \M$.

\item
 $a_i( 0, 0)>0, \limsup\limits_{x\to\infty}a_i( x, 0)<0$;\quad $b_i(0, 0)>0, \limsup\limits_{y\to\infty}b_i(0, y)<0$.
\end{enumerate}
}\end{asp}

\begin{asp}\label{asp5.2} {\rm
Every solution starting in $\R^2_+\setminus\{(0,0)\}$ will eventually enter an invariant set  $D\subset[0,H_0]^2\setminus [0, H_0^{-1}]^2$ where $H_0>1$ satisfying
$a_i(x, 0), b_i(0, y)>0$ if $x, y<H_0^{-1}$ and $a_i(x, 0), b_i(0, y)<0$ if $x, y>H_0$.
}\end{asp}

Consider two equations on the boundary
\begin{align}\label{e5.5}
\dot u(t)=u(t)a({\xi(t)}, u(t),0),\quad u(0)\in (0,\infty)\\
\label{e5.6}
\dot v(t)=v(t)b({\xi(t)}, 0,v(t)),\quad v(0)\in (0,\infty).
\end{align}

It is known that under Assumptions \ref{asp5.1} and \ref{asp5.2}, the Markov processes $(\xi(t), u(t))$ and  $(\xi(t), v(t))$ have unique invariant probability measures $\mu(\cdot)$ and $\nu(\cdot)$ respectively. We refer to \cite{DD} for the expression of the density functions of $\mu(\cdot)$ and $\nu(\cdot)$.
Like \eqref{e2.3} and \eqref{e2.4}, we define two values.
\begin{equation}\label{e5.7}
\bar\lambda_1=\sum_{i=1}^2\int\limits_{\R_+} b_i(u, 0)\mu(\{i\}\times du),\qquad
\bar\lambda_2=\sum_{i=1}^2\int\limits_{\R_+} a_i(0, v)\nu(\{i\}\times dv).
\end{equation}
In \cite{DDY1}, we showed that if $\bar\lambda_1$ and $\bar\lambda_2$ are positive, the process $(\xi(t), x(t), y(t))$ has an invariant probability measure in $\R^{2,\circ}$ that is unique and has some nice properties under additional assumptions. The goal of this section is to provide some results for \eqref{e5.1} when $\lambda_1$ and/or $\lambda_2$ are negative.
Let $z_{i_0, z_0}(t)=(x_{i_0, z_0}(t), y_{i_0, z_0}(t))$ be the solution to \eqref{e5.1} with initial value $\xi(0)=i_0, z_{i_0, z_0}(0)=z_0=(x_0,y_0)$. Denote by
$u_{i_0, x_0}(t)$ and $v_{i_0, y_0}(t)$ the solutions to \eqref{e5.2} and \eqref{e5.3} with initial value
$\xi(0)=i_0, u_{i_0, x_0}(0)=x_0$ and $\xi(0)=i_0, v_{i_0, y_0}(0)=y_0$ respectively.
In view of Assumption \ref{asp5.2}, we assume in the sequel that $z_{i_0, z_0}(t), u_{i_0,x_0}(t), v_{i_0, y_0}(t)\in[0,H_0]^2\setminus[0,H_0^{-1}]^2\,\forall t\geq0.$
\begin{prop}\label{prop5.1}
If $\bar\lambda_1<0$,
for any $\eps, \gamma>0, \lambda\in (0,-\bar\lambda_1)$, there is a $\bar\delta>0$ such that for all $ (i_0, z_0)\in\M\times\Big(D\cap[H_0^{-1},H_0]\times(0,\bar\delta)\Big),$
$$\PP\left(\left\{\limsup\limits_{t\to\infty}\dfrac{\ln y_{i_0, z_0}(t)}t\leq-\lambda\right\}\cap\bigg\{\left|u_{i_0, x_0}(t)-x_{i_0, z_0}(t)\right|\leq\gamma\,\forall t\geq0\bigg\}\right)\geq 1-\eps.$$
\end{prop}

\begin{proof}
Since $(\xi(t), u(t))$ is an ergodic Markov process,
the result  can be proven in the same manner as in Proposition \ref{prop4.1}.
It should be noted that it is even
simpler to have such results for \eqref{e5.1} than for \eqref{e1.2} because of two reasons.
First, some estimates for \eqref{e5.1} can be done with probability 1 in view of the nature of a piecewise deterministic process. Second, under Assumption \ref{asp5.2},  the solution of \eqref{e5.1}  evolves only in a compact domain.
The only difference that should be pointed out is that we do not compare $u_{i_0,x_0}^{-1}(t)$ and
$x_{i_0, z_0}^{-1}(t)$ like Proposition \ref{prop4.1}. Instead, we compare $\ln (u_{i_0,x_0}(t))$ and $\ln (x_{i_0, z_0}(t))$.
Since  $z_{i_0, z_0}(t)\in [0,H_0]^2\setminus[0,H_0^{-1}]^2\,\forall\,t\geq0,$
if $y_{i_0, z_0}(t)<H_0^{-1}$ then $x_{i_0, z_0}(t)\in[H_0^{-1},H_0]$ and
\begin{equation}\label{e5.9}
H_0^{-1}|u_{x_0}(t)-
x_{i_0, z_0}(t)|\leq \big|\ln u_{i_0,x_0}(t)-\ln x_{i_0, z_0}\big|\leq H_0|u_{i_0,x_0}(t)-
x_{i_0, z_0}(t)|.
\end{equation}
From (1) of Assumption \ref{asp5.1}, there is a  $\kappa>0$  such that $\dfrac{\partial a_i(x,
0)}{\partial x}\leq -\kappa\,\forall\,\delta\leq x\leq H_0$.
Let $\bar K=\sup\big\{\big|\dfrac{\partial a_i(x,y)}{\partial y}\big|: i\in\M, (x,y)\in[0,H_0]^2\big\}.$
It is clear from the mean value theorem that
\begin{equation}\label{e5.8}
\begin{aligned}
\dfrac{d}{dt}&\big(\ln u_{i_0,x_0}(t)-\ln x_{i_0, z_0}(t)\big)^2\\
&=\big(\ln u_{i_0,x_0}(t)-\ln x_{i_0, z_0}(t)\big)\Big[a({\xi(t)}, u_{i_0,x_0}(t), 0)-
a({\xi(t)}, x_{i_0, z_0}(t), y_{i_0, z_0}(t))\Big]\\
&\leq-\kappa\big(\ln u_{i_0,x_0}(t)-\ln x_{i_0, z_0}(t)\big)\big(u_{i_0,x_0}(t)-x_{i_0, z_0}(t)\big)+\bar K y_{i_0, z_0}(t)\big|\ln u_{i_0,x_0}(t)-\ln x_{i_0, z_0}(t)\big|\\
&\leq -\dfrac{\kappa}{H_0}\big(\ln u_{i_0,x_0}(t)-\ln x_{i_0, z_0}(t)\big)^2+\bar K y_{i_0, z_0}(t)\big|\ln u_{i_0,x_0}(t)-\ln x_{i_0, z_0}(t)\big|\\
&\leq -\kappa_1\big(\ln u_{i_0,x_0}(t)-\ln x_{i_0, z_0}(t)\big)^2+\kappa_2 y^2_{i_0, z_0}(t)
\end{aligned}
\end{equation}
where $\kappa_1, \kappa_2$ are some positive constants.
From \eqref{e5.9} and \eqref{e5.8}, we can easily proceed like Proposition \ref{prop4.1} to obtain the desired result.
\end{proof}
We denote by $\pi^1_t(u,v)=(x_1(t,u,v),y_1(t,u,v)),$ (resp. $\pi^2_t(
u,v)=(x_2(t,u,v),y_2(t,u,v))$ the solution of Equation
$\eqref{e5.2}$ (resp. \eqref{e5.3}) with initial value $(u,v)$.
Because of the degeneracy of \eqref{e5.1}, we cannot obtain the counterparts of Theorems \ref{thm2.2} and \ref{thm2.3} for \eqref{e5.1} in general.
However,  such results can be achieved in some cases.

\subsection{Case Study 1}
We consider the case when one of the two systems \eqref{e5.2} and \eqref{e5.3} has a globally asymptotically
stable equilibrium that is  positive.

\begin{thm}\label{thm5.1}
Let Assumptions \ref{asp5.1} and \ref{asp5.2} be satisfied.
Assume that system \eqref{e5.2} has a globally
stable positive equilibrium $(x^*_1, y^*_1)$.
Let
\begin{equation}\label{e5.10}
S=\left\{(x, y)=\pi_{t_n}^{\varrho(n)}\circ\cdots\circ\pi_{t_1}^{\varrho(1)}(x^*_1, y^*_1):0<t_1,t_2,... ,t_n; \; n\in \N\right\},
\end{equation}
where $\varrho(k)=1$ if $k$ is even, otherwise $\varrho(k)=2$.
Let $\bar S$ be the closure of $S$ in $\R^2_+=\{(x,y): x\geq0, y\geq 0\}$.
Then we have
\begin{enumerate}
\item If $\bar\lambda_1<0$, $\bar\lambda_2>0$ and $\bar S\cap \{(x,y): y=0\}\ne\emptyset$
then $\PP\left\{\lim\limits_{t\to\infty}\dfrac{\ln y(t)}t=\bar\lambda_1\right\}=1$.
\item If $\bar\lambda_1<0$, $\bar\lambda_2<0$ and $\bar S\cap \{(x,y): y=0\}\ne\emptyset$, $\bar S\cap\{(x,y): x=0\}\ne\emptyset$
then $p_{i_0,z_0}, q_{i_0,z_0}>0$ and $p_{i_0,z_0}+q_{i_0,z_0}=1$
where $p_{i_0,z_0}=\PP\left\{\lim\limits_{t\to\infty}\dfrac{\ln y_{i_0,z_0}(t)}t=\bar\lambda_1\right\}$ and
$q_{i_0,z_0}=\PP\left\{\lim\limits_{t\to\infty}\dfrac{\ln x_{i_0,z_0}(t)}t=\bar\lambda_2\right\}$.
\end{enumerate}
\end{thm}

\begin{proof}
We shall only prove
claim 1 because the other one can be
 obtained with a slight modification. In view of \cite[Theorem 2.1]{DD}, if $\lambda_2>0$ then there is $\theta\in(0, H_0)$ such that for any initial value in $\M\times\R^{2,\circ}_+$, the process $(\xi(t), x(t), y(t))$ is recurrent relative to $\M\times D_\theta$  where $D_\theta=([\theta,H_0]\times(0,H_0])\setminus[0,H_0^{-1}]^2$. If  $\bar S\cap \{(x,y): y=0\}\ne\emptyset$, for any $0<\eps<H_0^{-1}$, there is $(x_1, y_1)\in S\cap \big([H^{-1}_0,H_0]\times(0,\eps)\big).$
Split $D_\theta$ into $K_\eps:=[\theta, H_0]\times[\eps, H_0]$ and $[H_0^{-1},H_0]\times(0,\eps)$.
Since $K_\eps:=[\theta, H_0]\times[\eps, H_0]$ is compact, it is proven implicitly in \cite[Theorem 2.2]{DD} that for any neighborhood $U_1$ of $(x_1, y_1)$ there is $T_{U_1}>0$ satisfying
$\inf\limits_{(i_0', z_0')\in\M\times K_\eps}\PP\{z_{i_0', z_0'}(T_{U_1})\in U_1\}>0$.
Let $U_1$ be such that $y<\eps\,\forall (x,y)\in U_1$, we claim that
$\inf\limits_{(i_0', z_0')\in\M\times K_\eps}\PP\{z_{i_0', z_0'}(T_{U_1})\in ([H_0^{-1},H_0]\times(0,\eps)\}>0$.
This estimate, combined with the recurrence relative to $\M\times D_\theta$ of $(\xi(t), x(t), y(t))$, yields that $(\xi(t), x(t), y(t))$ is  recurrent relative to $\M\times([H^{-1}_0,H_0]\times(0,\eps)$ for any initial value in $\M\times\R^{2,\circ}_+$.
In view of the strong Markov property of $(\xi(t), x(t), y(t))$ and Proposition \ref{prop5.1}, we can obtain  claim 1 of Theorem \ref{thm5.1}.
\end{proof}

\subsection{Case Study 2}
 Note that in view of Assumption \ref{asp5.1},
  there are   unique
 pairs $(u_1,u_2)$ and $(v_1, v_2)$  satisfying $a_i( u_i
, 0) = 0$ and $b_i(0, v_i)=0$
for $i=1,2$.
We now consider the case
 that each of the two species dominates a state.
  We describe this situation by the following assumption.

\begin{asp}\label{asp5.3} {\rm
$(0, v_1)$ (reps $(u_2, 0)$) is
 {a} saddle point of  system \eqref{e5.2}
(resp. \eqref{e5.3}) while $(u_1, 0)$ (resp. $(0, v_2)$) is stable.
Moreover, all positive solutions to \eqref{e5.2} (resp. \eqref{e5.3})
converge to the stable equilibrium $(u_1, 0)$ (resp. $(0, v_2)$).
}\end{asp}

By the center
 manifold theorem and the attractiveness of $(u_1, 0)$ and
$(0, v_2)$,
there exist $(x_1^\diamond, y_1^\diamond)$ and $(x_2^\diamond, y_2^\diamond)$  such that the solution to \eqref{e5.2} starting at
$(x_1^\diamond, y_1^\diamond)$ as well as the solution to \eqref{e5.3} starting at
$(x_2^\diamond, y_2^\diamond)$   can expand to the whole real line
and
\begin{equation}
\lim\limits_{t\to\infty}\pi_t^1(x_1^\diamond, y_1^\diamond)=(u_1, 0)
\mbox{ and }\lim\limits_{t\to-\infty}\pi_t^1(x_1^\diamond, y_1^\diamond)=(0, v_1)
\end{equation}
\begin{equation}
\lim\limits_{t\to\infty}\pi^2_t(x_2^\diamond, y_2^\diamond)=(0, v_2)
\mbox{ and }
\lim\limits_{t\to-\infty}\pi^2_t(x_2^\diamond, y_2^\diamond)=(u_2, 0).
\end{equation}
Denote  by $\Gamma_1$ and $\Gamma_2$ their orbits, respectively.

It is proved in \cite[Section 5]{DDY1} that for any compact set $K\subset\R^{2,\circ}_+$ and any neighborhood $U_2$ of $(x_2, y_2)\in \Gamma_1\cup\Gamma_2$, there is  $T_{U_2}>0$ such that
$\inf\limits_{(i_0', z_0')\in\M\times K}\PP\{z_{i_0', z_0'}(T_{U_2})\in ([H_0^{-1},H_0]\times(0,\eps)\}>0$.
Moreover, the closures of $\Gamma_1$ and $\Gamma_2$ have non-empty intersections with each of the two axes.
By these facts and using the argument in the proof of Theorem \ref{thm5.1}, we obtain the following result.

\begin{thm}\label{thm5.2}
Let Assumptions \ref{asp5.1}, \ref{asp5.2} and \ref{asp5.3} be satisfied.
\begin{enumerate}
\item If $\bar\lambda_1<0$, $\bar\lambda_2>0$
then $\PP\left\{\lim\limits_{t\to\infty}\dfrac{\ln y(t)}t=\bar\lambda_1\right\}=1$.
\item If $\bar\lambda_1<0$, $\bar\lambda_2<0$,
then $p_{i_0,z_0}>0, q_{i_0,z_0}>0$ and $p_{i_0,z_0}+q_{i_0,z_0}=1$
where $p_{i_0,z_0}=\PP\left\{\lim\limits_{t\to\infty}\dfrac{\ln y_{i_0,z_0}(t)}t=\bar\lambda_1\right\}$ and
$q_{i_0,z_0}=\PP\left\{\lim\limits_{t\to\infty}\dfrac{\ln x_{i_0,z_0}(t)}t=\bar\lambda_2\right\}$.
\end{enumerate}
\end{thm}

\section{Discussion}\label{sec:dis}
In this paper, we have provided sufficient conditions for coexistence as well as exclusion of a stochastic competitive Lotka-Volterra system \eqref{e1.2}. In fact, our conditions are very close to necessary ones.
 Only critical case when $\lambda_1=0$ or $\lambda_2=0$ has not been studied.
Let us return to \eqref{e1.2} where $B_i(\cdot), i=1,2,3$ may be correlate. To be more precise, we assume that $$(B_1(\cdot), B_2(\cdot), B_3(\cdot))^\top=A(W_1(\cdot), W_2(\cdot), W_3(\cdot))^\top,$$
where $W_i(\cdot), i=1,2,3$ are mutually independent Brownian motions and
$A$ is a constant $3\times3$ matrix with $1\leq$ rank$(A)$ $\leq 3$.
Equation \eqref{e1.2} on the $x$-axis and the $y$-axis becomes
\begin{equation}\label{e6.1}
d\varphi (t)=\varphi (t)\big(a_1-b_1\varphi (t)\big)dt+(\gamma_1\varphi(t)+\alpha_1\varphi^2 (t))dB_1(t)
\end{equation}
and
\begin{equation}\label{e6.2}
d\psi (t)=\psi (t)\big(a_2-b_2\psi (t)\big)dt+(\gamma_2\psi(t)+\alpha_2\psi^2 (t))dB_3(t),
\end{equation}
respectively.
We can verify the conditions of \cite[Theorem 3.1,
p. 447]{IW}
enables us to prove that
if $a_1-\dfrac{\gamma_1^2}2<0$, then
$\PP\{\lim\limits_{t\to\infty}\varphi(t)=0\}=1$ for all positive solutions $\varphi(t)$.
We can therefore use arguments similar to the proof of Proposition \ref{prop4.1} to show that if the initial value is close to be on the $x$-axis, the solution will converge to   the $x$-axis with an arbitrarily large probability.
In case  $a_1-\dfrac{\gamma_1^2}2>0$, \eqref{e6.1} has a unique invariant probability measure
whose density $\tilde f^*_1$ can be solved from the Fokker-Planck equation.
Define $$\tilde\lambda_1=a_2-\int_0^\infty\big(c_2\phi+\dfrac{\beta_2^2}2\phi\big)\tilde f^*_1(\phi)d\phi\text{ if }a_1-\dfrac{\gamma_1^2}2>0.$$
The value $\tilde\lambda_2$ can be defined in the same manner if $a_2-\dfrac{\gamma_2^2}2>0$.
Using our method introduced in Sections \ref{sec:coe} and \ref{sec:com} with slight modifications to treat extra terms, we can show that
if $\tilde\lambda_1, \tilde\lambda_2,>0$ then \eqref{e1.1} has an invariant probability measure in $\R^{2,\circ}_+.$
If $\tilde\lambda_1<0$, the result in Proposition \ref{prop4.1} holds for \eqref{e1.1}.
For this reason, if the diffusion in \eqref{e1.1} is nondegenerate, the results stated in Theorems \ref{thm2.1}, \ref{thm2.2}, and \ref{thm2.3} hold for \eqref{e1.1} with $\lambda_1,\lambda_2$ replaced by $\tilde\lambda_1, \tilde\lambda_2.$
The convergence to the boundary in case either $a_1-\dfrac{\gamma_1^2}2$ or $a_2-\dfrac{\gamma_2^2}2$ is negative can also obtained.
If the diffusion is degenerate, we need to investigate the Lie-algebra generated by the drift and the diffusion as well as the corresponding control system to get further results under some additional assumptions. The reader might find how generate Lotka-Volterra models of predator-prey type are treated in \cite{DDT, RR, DDY2} in light of well-known results in \cite{WK, IK, SV}.

{\color{blue}
As a special case, when $\alpha_i=\beta_i=0, \gamma_i\ne 0, i=1,2,$
\eqref{e1.1} becomes
\begin{equation}\label{e6.3}
\begin{cases}
dX(t)=X(t)\big(a_1-b_1X(t)-c_1Y(t)\big)dt+\gamma_1X(t)dB_1(t)\\
dY(t)=Y(t)\big(a_2-b_2Y(t)-c_2X(t)\big)dt+\gamma_2Y(t)dB_3(t).
\end{cases}
\end{equation}
In this case, it is easy to compute $\tilde\lambda_i, i=1,2$.
In fact,
$$
\tilde\lambda_1=a_2-\dfrac{c_2}{b_1}\left(a_1-\dfrac{\gamma_1^2}2\right)\,\text{ if }\, a_1-\dfrac{\gamma_1^2}2>0;
\,\qquad
\tilde\lambda_2=a_1-\dfrac{c_1}{b_2}\left(a_2-\dfrac{\gamma_2^2}2\right)\,\text{ if }\, a_2-\dfrac{\gamma_2^2}2>0.
$$
Assuming that $B_1(t)$ and $B_3(t)$ are independent standard Brownian motions,
applying the
results for \eqref{e1.1} to the special case \eqref{e6.3}, we have the following assertions.

\begin{thm}\label{thm6.1}
Let $Z_{z_0}(t)=(X_{z_0}(t), Y_{z_0}(t))$ be the solution to \eqref{e6.3} with initial value $z_0\in\R^{2,\circ}_+$.
Then the following assertions hold:
\begin{enumerate}[{\rm 1.}]
\item
If $a_1-\frac{\gamma_1^2}2<0$ then $X_{z_0}(t)$ converges to $0$ almost surely with an exponential rate $a_1-\frac{\gamma_1^2}2$.
\item If $a_2-\frac{\gamma_2^2}2<0$ then $Y_{z_0}(t)$ converges to $0$ almost surely with an exponential rate $a_2-\frac{\gamma_2^2}2$.
\item If $a_i-\frac{\gamma_i^2}2\geq0, i=1,2$ and $\tilde\lambda_1>0$, $\tilde\lambda_2<0$ then $X_{z_0}(t)$ converges to $0$ almost surely with an exponential rate $\tilde\lambda_2$.
\item If $a_i-\frac{\gamma_i^2}2\geq0, i=1,2$ and $\tilde\lambda_1<0$, $\tilde\lambda_2>0$ then $Y_{z_0}(t)$ converges to $0$ almost surely with an exponential rate $\tilde\lambda_1$.
\item If $a_i-\frac{\gamma_i^2}2\geq0, i=1,2$ and $\tilde\lambda_1>0$, $\tilde\lambda_2>0$ then the distribution of $Z_{z_0}(t)$ converges in total variation to an invariant probability measure on $\R^{2,\circ}_+$.
\item If $a_i-\frac{\gamma_i^2}2\geq0$ and $\tilde\lambda_i>0$, $i=1,2$ then
 for any $z_0\in\R^{2\circ}_+$, we have
$p_{z_0}>0, q_{z_0}>0$ and $p_{z_0}+q_{z_0}=1$ where
$$p_{z_0}=\PP\left\{\lim\limits_{t\to\infty} \dfrac{\ln X_{z_0}(t)}t=\tilde\lambda_2\right\}\text{ and } q_{z_0}=\PP\left\{\lim\limits_{t\to\infty} \dfrac{\ln Y_{z_0}(t)}t=\tilde\lambda_1\right\}.$$
\end{enumerate}
\end{thm}

This theorem recovers the main findings in \cite[Theorems 9 and 10]{LWW}.
Similar results in a slight different context can also be found in \cite{EHS}.
It indicates that our results generalizes existing ones to more complex models.

Many existing works have been devoted to studying stochastic ecological models.
However, most of them dealt with models with linear diffusion parts.
Our paper introduced a new approach to treating stochastic models with non-linear diffusion parts.
In particular, the techniques developed in this paper are suitable to treat generalizations of some existing stochastic ecological models such as
cooperative
models in \cite{LW}, predator-prey
models in \cite{DDY2, RR}, as well as food chain
models in
\cite{LB}.
It should be noted that our main idea relies on
analyzing the behavior of solutions on the boundary.
The model in this paper is two dimensional, so
we can explicitly compute the ergodic invariant probability measures on the boundary
as well as Lyapunov exponents $\lambda_1$ and $\lambda_2$.
In general, with the same idea and some modifications and developments in techniques, we can treat
stochastic models in higher dimensions.
The signs of Lyapunov exponents with respect to ergodic invariant probability measures
on the boundary determine the behavior of solutions in the interior domain.
In a higher dimension, we are in general unable to compute invariant probability measures explicitly,
so Lyapunov exponents may not be calculated explicitly.
However, they can be estimated via a numerical method.
More details would be given in the future.
}

 One may also consider a more general model with regime-switching. It means that the coefficients
 $a_i, b_i, c_i, \alpha_i, \beta_i, \gamma_i$, $i=1,2$ in \eqref{e1.1} are functions of a Markov chain $r(t)$ with finite states. We suppose that  $r(\cdot)$ is independent of $B_i(\cdot)$ $i=1,2,3$.
 If the generator of $r(\cdot)$ does not depend on the state of $Z(t)$, we can prove the existence and uniqueness of invariant probability measures on the $x$-axis and the $y$-axis.
Then, we can also define $\lambda_1,\lambda_2$ and obtain similar results without any difficulty.
However, if  the generator of $r(\cdot)$ is state-dependent
(that is, the switching depends on the diffusions),
the comparison  between solutions on the boundary and
those in the interior is much more difficult.
This deserves  more careful thoughts and consideration.

{\color{blue}Recently, stochastic ecosystems with delay have also been studied intensively (see e.g., \cite{BM, LB2} and references therein).
Although the main idea of this paper may work with delay systems, it is not easy to apply our method to those systems
directly.
The main difficulty is that we need to work with infinite dimensional function spaces
that are not locally compact.
It is thus difficult to obtain
certain uniform estimates.
It appears that
novel techniques are
needed to treat those models.
}

\end{document}